\documentclass[11pt]{article}

\usepackage[letterpaper, margin=1in]{geometry}
\usepackage{graphicx} 

\usepackage{amsmath}
\usepackage{amsthm}
\usepackage{thmtools}
\usepackage{amssymb}

\usepackage{comment}
\usepackage{dirtytalk}
\usepackage{dsfont}

\usepackage{tikz}
\usepackage{pgfplots}
\pgfplotsset{compat=1.18}

\usepackage{caption}
\usepackage[backend=biber, style=alphabetic]{biblatex}
\DeclareFieldFormat[article]{title}{\mkbibemph{#1}}
\usepackage{aliascnt}
\usepackage[noabbrev]{cleveref}

\crefname{equation}{}{}
\Crefname{equation}{Equation}{Equations}

\theoremstyle{definition}

\newtheorem{definition}{Definition}[section]

\newtheorem{remark}[definition]{Remark}

\theoremstyle{remark}

\newtheorem{example}[definition]{Example}

\theoremstyle{plain}

\newtheorem{assumption}{Assumption}[section]

\newtheorem{theorem}[definition]{Theorem}

\newtheorem{proposition}[definition]{Proposition}

\newtheorem{lemma}[definition]{Lemma}

\newtheorem{corollary}[definition]{Corollary}

\DeclareMathOperator{\id}{id}
\DeclareMathOperator{\Int}{Int}
\DeclareMathOperator{\CBB}{CBB}

\newcommand{\R}{\mathbb{R}}
\newcommand{\N}{\mathbb{N}}
\newcommand{\g}{\overline{g}}
\newcommand{\M}{M_\lambda}
\newcommand{\Mk}{M_{\lambda_k}}

\title{A Correspondence between Billiards and Geodesics}
\author{Daniele Giannetto\footnote{Corresponding author: Daniele Giannetto; Email: daniele.giannetto@mail.utoronto.ca}}
\date{June 2026}

\begin{document}

\maketitle

\begin{abstract}
From a geometric viewpoint, billiard trajectories and geodesics are related by mutual approximation results. In one direction, it is known that every geodesic curve on the boundary of a smooth convex body can be approximated by a sequence of billiard trajectories inside of it. We establish the other direction by proving that, for Riemannian billiard tables (under mild assumptions), there are families of fold-type surfaces such that every sequence of geodesic segments on these surfaces has a subsequence that converges to a billiard trajectory in the table. In particular, this is true for convex Euclidean tables. We also describe a more general class of tables for which this result holds and present explicit non-convex examples.
\end{abstract}

\tableofcontents

\section{Introduction}

A billiard trajectory in a Euclidean domain $K\subset \R^n$ with a smooth boundary $\partial K$ is defined as a piecewise straight line in the interior of $K$: upon reaching the boundary, it undergoes a reflection satisfying the rule that the angle of incidence equals the angle of reflection. 
It is well known that billiard dynamics in $K$ highly depends on the shape of $\partial K$; it can vary from being extremely rich to being not well-defined everywhere (see \cite{tabac}). Definitions of generalized billiards can be given in ambient non-Euclidean or pseudo-Riemannian geometries (see for instance \cite{khtab}) and for non-smooth tables (see \cite{Lange}). 
In this paper, we consider only billiard trajectories whose intersections with the boundary of the table, whenever they occur, are transverse; in particular, we exclude trajectories containing nontrivial segments in the boundary. We refer to \cite{Kourganoff} for a discussion of billiard trajectories that may be partially or entirely contained in the boundary of the table. 
For billiard tables with smooth boundary, by restricting attention to trajectories whose reflections at the boundary are transverse, the billiard evolution is well-defined and uniquely determined by the initial data.

From a geometrical perspective, billiard maps are closely related to geodesic flows. On the one hand, if $K\subset \R^n$ is a closed strictly convex Euclidean domain with sufficiently smooth boundary, it is always possible to approximate geodesics in the boundary $\partial K$ by means of billiard trajectories that live inside $K$. 
Namely, one can consider billiard trajectories with incidence angle at a given point of the boundary tending to zero, i.e. with the incoming velocity vector of the billiard trajectory approximating a given tangent vector of a geodesic on the boundary. More formally, the following result was proved in \cite{Lange}.

\begin{theorem}[Theorem B of \cite{Lange}]\label{thm:lange_thm}
  
Let $K$ be a convex body in $\R^n$ whose boundary is of class $C^{2,1}$ and has a positive definite second fundamental form at every point. Let $p_0\in \partial K$ and $\{\gamma_k\}_{k\in \N}$ be a sequence of billiard trajectories of $K$ starting from $p_0$. Then, if the initial direction of $\gamma_k$ converges to a tangent vector $v\in T_{p_0}(\partial K)$, we have that the sequence $\{\gamma_k\}_{k\in \N}$ converges locally uniformly to the corresponding geodesic of the boundary.

\end{theorem}

For the sake of completeness of this paper, a proof sketch of this theorem is presented in Appendix B.

On the other hand, billiard trajectories can also be regarded as limits of sequences of geodesic segments. For instance, a billiard trajectory in the closed unit disk $\overline D\subset \R^2$ can be regarded as the limit of a sequence of geodesic segments, each lying on an ellipsoid of revolution whose one semiaxes tends to 0 (see \cite{birkhoff}). More formally, let $K = \{(x,y,0)\in \R^3: \ x^2 +y^2\leq 1\}$ be the closed unit disk in the $xy$-plane of the Euclidean space $\R^3$. For each $\lambda\in (0,1)$, we let
$$M_\lambda =\left\{(x,y,z)\in \R^3:\ x^2 + y^2 + \frac{z^2}{\lambda^2} = 1\right\}$$
be the ellipsoid whose $x$ and $y$ semiaxes have length $1$ and whose $z$ semiaxis vanishes as $\lambda\rightarrow 0^+$ (see \Cref{fig:ellipsoids}). Then, one can prove (see \Cref{cor:main_corollary_sec3}) that every billiard trajectory that touches $\partial K$ transversally can be recovered as a cluster point of a family of geodesic segments $\{\gamma_\lambda: [-T,T]\rightarrow M_\lambda\}_{\lambda\in (0,1)}$.

\begin{figure}
  \includegraphics[width=5.2cm]{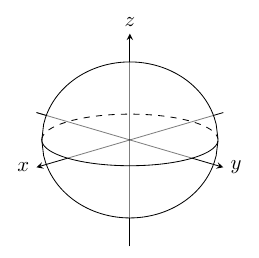}
  \hfill
  \includegraphics[width=5.2cm]{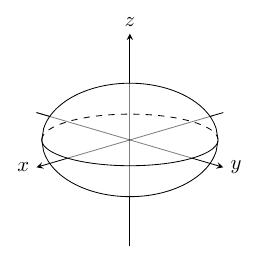}
  \hfill
  \includegraphics[width=5.2cm]{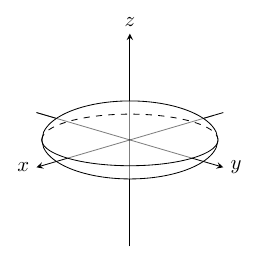}
\caption{Ellipsoids in $\R^3$ with the $x$ and $y$ semiaxes of unit length and with the $z$-semiaxes of length: 1, 0.7, 0.4 respectively.}\label{fig:ellipsoids}
\end{figure}

The purpose of this paper is to prove that a similar construction and convergence result can be carried out for a general Riemannian billiard table under mild curvature assumptions. Recent advances in this direction can be found in \cite{Kourganoff}, in which the author treats extensively the case of two-dimensional billiard tables and explores hyperbolic type properties of the billiard dynamics.

In what follows, we briefly describe the structure of the present work.

Section 2 establishes notations and collects well-known results and constructions from the theory of metric geometry that will be important in section 3. In particular, we will focus on the notion of quasigeodesic curves in Riemannian manifolds with boundary that satisfy a lower curvature bound. Those curves are strongly related to both geodesics and billiard trajectories. Details and proofs for the results presented in section 2 can be found in \cite{PetruninSemiconcave}, \cite{Course}, \cite{Lange} and references therein.\vspace{0.1cm}

Section 3 is the core of the present work. First, we give rigorous definitions of a billiard table $K$, namely, the closure of an open connected subset of $\R^n$ endowed with the Riemannian metric induced by the ambient space and having a boundary of class $C^{2,1}$, and of a billiard trajectory $\gamma:[-T,T]\to K$ within it. In this paper, a billiard trajectory is defined as a concatenation of finitely many locally length-minimizing segments satisfying the billiard reflection law whenever two consecutive segments meet at the boundary of the table. We emphasize that each segment is parameterized by arclength. \Cref{thm:quasigeodesic_characterization} describes in detail the relation between billiard trajectories and quasigeodesic curves introduced in section 2. Next, if $H\subset \R^{n+1}$ is a hyperplane and $K\subset H$ is a billiard table, we construct a sufficiently general notion of family of folds over $K$, denoted $(M_\lambda,g_\lambda)$, $\lambda\in (0,1)$, which is made of hypersurfaces that \say{flatten} onto some portion of $K$ in a fashion similar to the family of ellipsoids presented above. We also provide an explicit example of such a family of folds, which will be the main focus of section 4. 
In this context, the following theorem summarizes the first main result of the paper.

\begin{theorem}\label{thm:convex_billiard_intro}
  
Let $K\subset H\subset \R^{n+1}$ be a convex Euclidean billiard table whose boundary is of class $C^{2,1}$. Then, there is $T>0$ such that every billiard trajectory $\gamma:[-T,T]\to K$ touching $\partial K$ at $p_0 := \gamma(0)$ transversally can be recovered as the uniform limit of a sequence of arclength-parameterized geodesic segments
$$\left\{\gamma_k:[-T,T]\rightarrow M_{\lambda_k}\right\}_{k\in \N}$$
with $\gamma_k(0)=p_0$ for all $k\in\N$ and where $\lambda_k\to 0^+$ as $k\to \infty$.
  
\end{theorem}

Notice that the above theorem generalizes the convergence result presented before for the family of ellipsoids with a vanishing semiaxis to every convex Euclidean billiard table with $C^{2,1}$ boundary. Also, this theorem can be generalized for certain non-convex billiard tables (at least in low dimension) and its proof is postponed to section 4.

Moreover, it turns out that the same result holds for a wider class of Riemannian billiard tables and families of folds, under at least one of the following extra assumptions:

\begin{assumption}\label{ass:assumption1_intro}
There exists a $\kappa\in \R$ such that, for every $\lambda\in (0,1)$, there is a compact neighborhood of $p_0$ in $M_{\lambda}$ which is a space of curvature bounded below by $\kappa$.
\end{assumption}

\begin{assumption}\label{ass:assumption2_intro}
There is $\tilde T>0$ such that, for every $\lambda\in (0,1)$, each geodesic segment $\gamma:[-\tilde T,\tilde T]\to M_\lambda$ with $\gamma(0) = p_0$ is $g_\lambda$-length minimizing.
\end{assumption}

Then, we prove the following results.

\begin{theorem}\label{thm:compactness_intro}
  Assume that the family $\{(M_\lambda,g_\lambda)\}$ satisfies at least one of \cref{ass:assumption1_intro} or \cref{ass:assumption2_intro}. Then, for each decreasing sequence $\{\lambda_k\}_{k\in \N}\subset (0,1)$ with $\lambda_k\to 0^+$, there exists $T>0$ such that every sequence of arclength-parameterized geodesic segments
  $$\left\{\gamma_k:[-T,T]\rightarrow \Mk\right\}_{k\in \N},$$
  that satisfies $\gamma_k(0)=p_0$ for all $k\in\N$, admits a subsequence $\{\gamma_m\}_{m\in \N}$ that converges to a continuous curve $\gamma:[-T,T]\rightarrow K$ uniformly in $[-T,T]$ as $m\rightarrow \infty$. Moreover, if $\gamma$ intersects $\partial K$ transversally, then it is a billiard trajectory of $K$.
\end{theorem}

\begin{theorem}\label{thm:main_theorem}

Assume that the family $\{(M_\lambda,g_\lambda)\}$ satisfies at least one of \cref{ass:assumption1_intro} or \cref{ass:assumption2_intro}. Then there is $T>0$ such that every billiard trajectory $\gamma:[-T,T]\to K$ touching $\partial K$ at $p_0 := \gamma(0)$ transversally can be recovered as the uniform limit of a sequence of arclength-parameterized geodesic segments
$$\left\{\gamma_k:[-T,T]\rightarrow M_{\lambda_k}\right\}_{k\in \N}$$
with $\gamma_k(0)=p_0$ for all $k\in\N$ and where $\lambda_k\to 0^+$ as $k\to \infty$.

\end{theorem}

Section 4 presents the proof of \Cref{thm:convex_billiard_intro} and an analogous result for a certain class of non-convex tables. 

Appendix A contains technical details of the construction of the particular family of folds introduced in section 3, while Appendix B contains a proof sketch of \Cref{thm:lange_thm}.

\begin{paragraph}{Strategy of the proof of \Cref{thm:main_theorem}}
The proof of \Cref{thm:main_theorem} is presented in section 3.3 and it is divided into two main steps. 

First, we prove \cref{thm:compactness_intro} using a standard compactness argument and we establish that each sequence of arclength-parameterized geodesic segments whose elements pass through $p_0$ and lie in different folds of the same family admits a subsequence that converges to a continuous curve of the billiard table. Moreover, if such a limit curve touches $\partial K$ transversally, then it must be a billiard trajectory. The conclusion is then obtained as a consequence of the preliminary results presented in section 2 (see \Cref{prop:geodesic_convergence}, \Cref{prop:quasigeodesic_limit}) and in section 3 (see \Cref{thm:quasigeodesic_characterization}). Notice that these results can be applied only provided that at least one between \Cref{ass:assumption1_intro} and \Cref{ass:assumption2_intro} holds.

Then, if we fix a billiard trajectory in $K$ that touches $\partial K$ transversally, we show that it is possible to construct a sequence of arclength-parameterized geodesic segments satisfying the assumptions of the first step which then converges uniformly to the fixed billiard trajectory (up to a subsequence). Here we emphasize the importance of starting with a trajectory that touches $\partial K$ transversally, as uniqueness of the billiard flow in $K$ plays a crucial role in the proof.

\end{paragraph}

\paragraph{Acknowledgements}
I am grateful to Boris Khesin for bringing this problem to my attention and for many fruitful discussions and helpful suggestions throughout the project. I thank Sergei Tabachnikov and Anton Petrunin for the helpful discussions and advice. 

\section{Preliminaries}

\subsection{Notation and conventions}

In the following definitions, let $\Omega\subseteq \R^n$ be a connected set with non-empty interior.

\begin{definition}

A function $f:\Omega\to \R$ is said to be \emph{of class $C^{2,1}$ on $\overline \Omega$}, denoted $f\in C^{2,1}(\overline \Omega)$, if $f$ is of class $C^2$ in $\Int(\Omega)$, with Lipschitz continuous second derivatives. Moreover, the function $f$ and all of its first and second derivatives extend continuously up to the boundary of $\Omega$. 

\end{definition}

\begin{definition}

We say that $\partial \Omega$ is \textit{of class $C^{2,1}$} if for every $x_0\in \partial \Omega$ there is an open bounded neighborhood $U$ of $x_0$ in $\R^n$ and a function $f\in C^{2,1}(\overline \Omega)$ such that
$$\overline{\Omega}\cap U = \{x\in U: f(x)\ge 0\}, \quad \partial \Omega \cap U = f^{-1}(0),$$
and $f$ has a regular value at $0$, namely $\nabla f(x) \neq 0$ for all $x\in f^{-1}(0)$.

\end{definition}

In what follows we let $(M,g)$, with $M\subset \R^N$, be a smooth orientable $n$-dimensional Riemannian manifold with or without boundary.

\begin{definition}\label{def:space_curvature_below}

We say that $(M,g)$ is a \textit{space of curvature bounded below by $\kappa\in \R$} if 
\begin{itemize}
  \item for every $p\in M\setminus \partial M$, all of the sectional curvatures of $M$ at $p$ are greater than or equal to $\kappa$, and
  \item for every $p_0\in \partial M$ and $u\in T_{p_0}(\partial M)$ we have 
      $$\mbox{II}_{p_0}^{\partial M}(u,u) \ge 0,$$
      where $\mbox{II}_{p_0}^{\partial M}$ is the scalar second fundamental form of the boundary $\partial M$ (viewed as a hypersurface in $M$) at $p_0$ associated to the unit normal field to $\partial M$ pointing into $M$.  
\end{itemize}

\end{definition}

\begin{definition}

Let $p_0\in\partial M$ be a boundary point and let $\nu(p_0)\in T_{p_0}M$ be the unique normal vector to $\partial M$ at $p_0$ pointing into $M$. The topological subspace
$$C_{p_0}(M):= \left\{\tilde w + r\ \nu(p_0):\ \tilde{w}\in T_{p_0}(\partial M),\ r\ge0\right\}\subset T_{p_0}M$$
is called the \emph{half tangent space of $M$ at $p_0$}.

\end{definition}

\begin{definition}\label{def:polar_vectors}
Let $p_0\in \partial M$. The unit vectors $u,v\in C_{p_0}(M)$ are called \textit{polar} if 
$$g_{p_0}(u+v,w) \ge 0$$ 
for each $w\in C_{p_0}(M)$. If $p_0$ is not a boundary point, then we say that two unit vectors $u,v\in T_{p_0}M$ are polar whenever they are opposite.
\end{definition}

\begin{remark}

If $p_0\in \partial M$ is a boundary point and $u = \tilde u + r\ \nu(p_0)\in C_{p_0}(M)$ is of unit norm, then there exists a unique unit vector $v = -\tilde u + r\ \nu(p_0)\in C_{p_0}(M)$ such that $u$ and $v$ are polar. In particular, if $u\in T_{p_0}(\partial M)$, then $v=-u\in T_{p_0}(\partial M)$ (see \Cref{fig:polar_vectors}).

\end{remark}

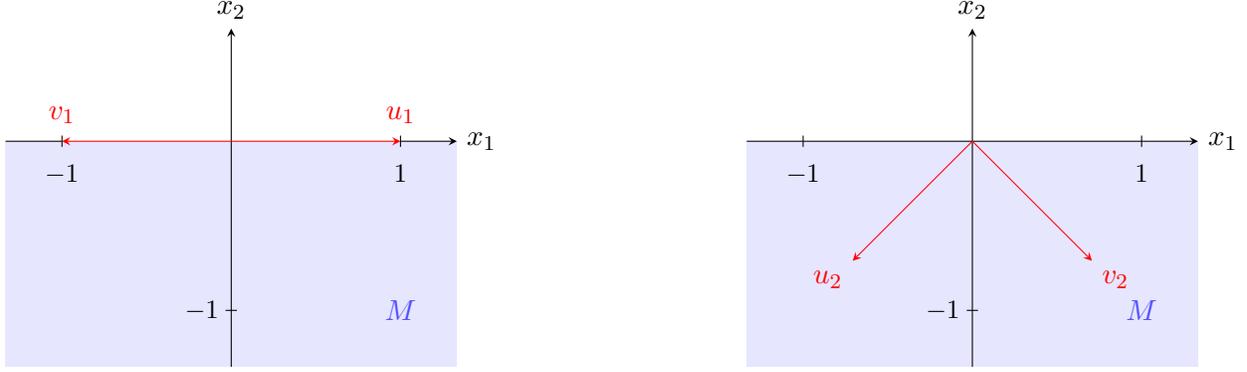
\begin{figure}
  \begin{tikzpicture}[>=stealth, scale=1.5]
  \fill[blue!10] (-2,-2) rectangle (2,0);
  
  \draw[->] (1.5,0) -- (2,0) node[right] {$x_1$};
  \draw (-2,0) -- (-1.5,0);
  \draw[->] (0,-2) -- (0,1) node[above] {$x_2$};

  \draw (1.5,0.05) -- (1.5,-0.05) node[below=3pt] {\small $1$};
  \draw (-1.5,0.05) -- (-1.5,-0.05) node[below=3pt] {\small $-1$};
  \draw (-0.05,-1.5) -- (0.05,-1.5) node[left=3pt] {\small $-1$};
  
  \draw[->, red] (0,0) -- (1.5,0) node[above=3pt] {$u_1$};
  \draw[->, red] (0,0) -- (-1.5,0) node[above=3pt] {$v_1$};
  
  \node[blue!70] at (1.5,-1.5) {$M$};
  \end{tikzpicture}
  \hfill
  \begin{tikzpicture}[>=stealth, scale=1.5]
  \fill[blue!10] (-2,-2) rectangle (2,0);
  
  \draw[->] (-2,0) -- (2,0) node[right] {$x_1$};
  \draw[->] (0,-2) -- (0,1) node[above] {$x_2$};

  \draw (1.5,0.05) -- (1.5,-0.05) node[below=3pt] {\small $1$};
  \draw (-1.5,0.05) -- (-1.5,-0.05) node[below=3pt] {\small $-1$};
  
  \draw (-0.05,-1.5) -- (0.05,-1.5) node[left=3pt] {\small $-1$};

  \draw[red,->, domain=0:1.5*sqrt(2)/2] plot(-\x,-\x) node[below left] {$u_2$};
  \draw[red,->, domain=0:1.5*sqrt(2)/2] plot(\x,-\x) node[below right] {$v_2$};
  
  \node[blue!70] at (1.5,-1.5) {$M$};
  \end{tikzpicture}
  \caption{$M=\{x_2\leq 0\} = C_{(0,0)}M\subset \R^2$. On the left, the pair of polar vectors $u_1 = (-1,0)$, $v_1 = (1,0)$ and, on the right, the pair of polar vectors $u_2 = (-\sqrt{2}/2,-\sqrt{2}/2)$, $v_2 = (\sqrt{2}/2,-\sqrt{2}/2)$.}\label{fig:polar_vectors}
\end{figure}

\begin{definition}\label{def:regular_curve}
Let $c:[a,b]\rightarrow M$ be a Lipschitz continuous function on $[a,b]\subset \R$. We say that $c$ is a \textit{regular curve in $M$} if the left and right derivatives $c_\pm^\prime(t)\in\R^N$ exist and are unique for every $t\in [a,b]$.
\end{definition}

\begin{definition}
Let $c:[a,b]\rightarrow M$ be a Lipschitz regular curve in $M$ and let $t\in [a,b]$. The vectors
$$c^+(t):= c^\prime_+(t),\qquad c^-(t):= -c^\prime_-(t)$$
are called the \textit{left and right tangent vectors to $c$ at $t$} respectively.
\end{definition}

\begin{remark}
If $c: [a,b]\rightarrow M$ is a Lipschitz regular curve in $M$, we have that if $c(t)\in \partial M$, then $c^+(t),c^-(t)\in C_{c(t)}(M)$.
\end{remark}

\subsection{The double of a Riemannian manifold with boundary}

Let $(M,g)$ be a smooth Riemannian manifold with non-empty boundary $\partial M$. The goal of this section is to recall the well-known notion of the double of $M$ along its boundary. Its construction and most of the properties we will recall can be found in \cite{LeeManifolds} and \cite{Course}.

Let $M_+$ and $M_-$ be two copies of $M$. If $x\in M$, we denote by $x_+$ and $x_-$ the corresponding points in $M_+$ and $M_-$. 

\begin{definition}
The quotient space
$$D(M):=(M_+\sqcup M_-)/\sim,$$
where $x_+\sim x_-$ if and only if $x\in \partial M$, is called the \emph{double of $M$ along $\partial M$}.
\end{definition}

\begin{remark}
Every point of $\Int(M)$ gives rise to two distinct points of $D(M)$, one in each copy of $M$, whereas every boundary point gives rise to a single point in $D(M)$.
\end{remark}

The quotient $D(M)$ comes with a folding map $\pi_M:D(M)\to M$, defined by $\pi_M(x_+) = \pi_M(x_-) = x$. Topologically, $D(M)$ is a smooth manifold without boundary. Indeed, around a point which is away from $\partial M$ a chart of $D(M)$ is given by a chart of one of the copies of $M$. 
Instead, at a boundary point $x_0\in \partial M$ we may construct a coordinate chart by gluing two copies of a boundary chart of $M$ as follows. Let $V\subset M$ be a boundary chart of $x_0$ with coordinates $(\theta,r)$, where $\theta$ is a local coordinate for $\partial M$ around $x_0$ and $r\in[0,\varepsilon)$. In $\tilde V = \pi_M^{-1}(V)$, we define the smooth coordinates $(\theta,s)$, where $\theta$ is as before and $s\in (-\varepsilon,\varepsilon)$ is such that 
$$s(q) = \begin{cases}
           r(\pi_M(q)), & \mbox{if } q\in M_+ \\
           -r(\pi_M(q)), & \mbox{if } q\in M_-.
         \end{cases}$$
Although smooth away from $\partial M$, in general the folding map $\pi_M:D(M)\to M$ is continuous but not even $C^1$ at some point $x_0\in \partial M$. Indeed, with respect to the doubled coordinates $(\theta,s)$ introduced before, its local expression is $\pi_M(\theta,s)=(\theta,|s|)$. Still, we are able to define its one-sided differentials at $x_0$ as a folding map
$$d_{x_0}\pi_M:T_{x_0}D(M)\to C_{x_0}M,\quad d_{x_0}\pi_M(\xi+\alpha \partial_r) = \xi + |\alpha|\nu(x_0),$$ 
where $\xi\in T_{x_0}\partial M$, $\alpha\in \R$ and $\nu(x_0)$ is the unit normal to $\partial M$ at $x_0$ pointing into $M$.

\begin{definition}
The quotient $D(M)$ is endowed with a canonical distance function $d_{D(M)}:D(M)\times D(M)\to \R$, defined by
$$d_{D(M)}(x_i,y_j)=\begin{cases}
                      d_g(x,y), & \mbox{if } i=j, \\
                      \inf \left\{(d_g(x,z)+d_g(z,y)):\ z\in \partial M\right\}, & \mbox{otherwise}.
                    \end{cases}$$ 
\end{definition}

The metric $d_{D(M)}$ may be described equivalently in two ways. First, the length functional induced by $d_{D(M)}$ is the one obtained by the pullback tensor $\pi_M^*g$. Such pullback is a positive semidefinite bilinear form at each $q\in D(M)$ which degenerates at points in $\partial M$ along a direction not tangent to $\partial M$. 

On the other hand, in doubled boundary normal coordinates $(\theta,s)\in \partial M\times (-\varepsilon,\varepsilon)$, the function $d_{D(M)}$ is associated to the Riemannian distance function given by the Lipschitz Riemannian metric
$$g_{D(M)}=ds^2+h_{|s|}.$$
where, in the same coordinates, we have $g = ds^2+h_s$, $s\ge 0$. Such a metric $g_{D(M)}$ is truly non-degenerate everywhere and smooth away from $\partial M$, but its coefficients are in general only Lipschitz across $\partial M$.

In what follows, we recall some important properties of the projection $\pi_M:D(M)\to M$.

\begin{proposition}\label{prop:double_projection}
  Let $(M,g)$ be a smooth Riemannian manifold with boundary and let\\ $(D(M),d_{D(M)})$ denote its double along $\partial M$. Then
  \begin{enumerate}
    \item $\pi_M$ is a $1$-Lipschitz map with respect to the distances $d_g$ and $d_{D(M)}$;
    \item if $x_0\in \partial M\subset D(M)$ and $u,v\in T_{x_0}D(M)$ are unit vectors which are polar with respect to $g_{D(M)}$, then $d_{x_0}\pi_M(u)$ and $d_{x_0}\pi_M(v)$ are unit vectors in $C_{\pi_M(q)}(M)$ which are polar with respect to $g$.
    \item if $\tilde c:[a,b]\to D(M)$ is a Lipschitz regular curve which is parameterized by arclength with respect to $d_{D(M)}$, then $\pi_M\circ \tilde c: [a,b]\to M$ is a Lipschitz regular curve which is parameterized by arclength with respect to $d_g$;
  \end{enumerate} 
\end{proposition}

\begin{proof}
We prove the three claims separately. First, let $x_i,y_j\in D(M)$, where $i,j\in\{+,-\}$. If $i=j$, then by definition
$$d_g(\pi_M(x_i),\pi_M(y_i)) =d_g(x,y)= d_{D(M)}(x_i,y_i).$$
If $i\neq j$, then
$$d_g(\pi(x_i), \pi(y_j)) =d_g(x,y)\leq \inf\{ d_g(x,z)+d_g(z,y): z\in \partial M\} = d_{D(M)}(x_i,y_j),$$
due to the triangle inequality. Hence $\pi_M$ is $1$-Lipschitz map.

Next, we prove the second claim. Let $x_0\in \partial M$ and choose boundary normal coordinates $(\theta,s)$ where $\theta$ is a local coordinate on $\partial M$ and $s\in (-\varepsilon,\varepsilon)$ is the normal coordinate to $\partial M$. Recall that, in these coordinates, the metrics $g$ and $g_{D(M)}$ have the forms
$$g=ds^2+h_s,
\quad s\in [0,\varepsilon),\qquad g_{D(M)}=ds^2+h_{|s|},
\quad s\in (-\varepsilon,\varepsilon).$$
Then, observe that if $u = \tilde u + \alpha\partial_s\in T_{x_0}D(M)$ is a unit vector, i.e. $(g_{D(M)})_{x_0}(u,u) = |\tilde u|_{h_0}^2 + |\alpha|^2 = 1$, the vector $d_{x_0}\pi_M(u) = \tilde u + |\alpha|\nu(x_0)$ has unit norm as well. Indeed
$$g_{x_0}(d_{x_0}\pi_M(u),d_{x_0}\pi_M(u)) = |\tilde u|_{h_0}^2 + |\alpha|^2 = 1.$$
Moreover, if $u = \tilde u + \alpha\partial_s$ and $v = \tilde v + \beta\partial_s$ are unit polar vectors in $T_{x_0}D(M)$ we have that, since $D(M)$ has no boundary, $u = -v$. But then, if $w = \tilde w + s\nu(x_0)\in C_{x_0}(M)$, we estimate
$$g_{x_0}\bigl(d_{x_0}\pi_M(u)+d_{x_0}\pi_M(v),w\bigr)= g_{x_0} \bigl(2|\alpha|\nu(x_0),\tilde w+s\nu(x_0)\bigr)=2|\alpha|s \geq 0.$$
Hence $d_{x_0}\pi_M(u)$ and $d_{x_0}\pi_M(v)$ are polar in $C_{x_0}M$.

Finally, we prove the third claim. Let $\tilde c:[a,b]\to D(M)$ be a Lipschitz regular curve parameterized by arclength, and define $c:=\pi_M\circ \tilde c$. Since $\pi_M$ is $1$-Lipschitz, $c$ is Lipschitz as well. It remains to prove that $c$ is regular and has unit speed. Away from $\partial M\subset D(M)$, the map $\pi_M$ is locally a smooth isometry from one of the two copies of $M$ onto $M$, and therefore the conclusion is immediate. Now, let $t_0\in[a,b]$ be such that $\tilde c(t_0)=x_0\in \partial M\subset D(M)$. Since tangent vectors of $\tilde c$ transform under $d_{x_0}\pi_M$ into tangent vectors of $c$, and since $d_{x_0}\pi_M$ sends unit vectors to unit vectors, the claim follows.
\end{proof}

\subsection{Riemannian comparison theory and quasigeodesic curves}

Let $(M,g)$, with $M\subset \R^N$, be a complete orientable smooth $n$-dimensional Riemannian manifold (with or without boundary) and let $d_g: M\times M \rightarrow \R$ denote the corresponding Riemannian distance function.  In what follows, we will assume that $(M,g)$ is a space of curvature bounded below by some $\kappa\in \R$. Also, we will let
$$\alpha_\kappa :=\begin{cases} \dfrac{\pi}{\sqrt{\kappa}} &\text{if}\ \kappa>0,\\  +\infty &\text{if}\ \kappa\leq 0\end{cases}$$
denote the diameter of the unique complete simply connected 2-dimensional surface of constant curvature $\kappa$ (see \cite{LeeRiemannian}).

Let $p\in M$ and $c:[a,b]\rightarrow M$ be an arclength-parameterized geodesic segment such that $0<d_{g}(p,c(t))< \alpha_\kappa$ \footnote{When $\kappa>0$, the condition $0<d_g(p,c(t))<\alpha_\kappa$ ensures that relevant comparison construction is well behaved.} for all $t\in [a,b]$. Then, the Hessian comparison theorem (see \cite[Theorem 11.7]{LeeRiemannian}) implies that the function
$$t\mapsto d_{g}(p,c(t))$$
satisfies the following concavity condition in the form of a differential inequality; see \cite{PetruninSemiconcave}:
\begin{equation}\label{eq:concavity_condition}
    f_p^{\prime\prime}(t)\leq 1- \kappa f_p(t), \quad \text{for all}\ t\in (a,b),
\end{equation}
where $f_p(t) = \rho_\kappa(d_{g}(p,c(t)))$ and 
\begin{equation*}
\rho_\kappa(r) =\begin{cases}
    \dfrac{1-\cos(r\sqrt{\kappa})}{\kappa}, & \mbox{if } \kappa>0, \\
    \dfrac{r^2}{2}, & \mbox{if } \kappa=0, \\
    \dfrac{1-\cosh(r\sqrt{-\kappa})}{\kappa}, & \mbox{if } \kappa<0.
  \end{cases}
\end{equation*}

\begin{remark}

Since the continuous function $f_p:[a,b]\rightarrow \R$ might not be twice differentiable everywhere in $(a,b)$, the differential inequality \cref{eq:concavity_condition} must be understood in the barrier sense\footnote{We refer to \cite{manteg} for further discussions on differential inequalities in the barrier sense and their applications to Riemannian geometry.} and it reads as follows. For every $\Phi\in C^2([a,b])$ that is a solution of $\Phi^{\prime\prime}(t) = 1-\kappa f_p(t)$ for all $t\in [a,b]$, the function
$$t\mapsto f_p(t)-\Phi(t)$$
is concave in $[a,b]$. If $f$ is piecewise $C^2$, then the above inequality holds in the barrier sense if and only if it holds in the classical sense on each smooth subinterval and, at every nonsmooth time $t_0$, one has $f'_-(t_0)\geq f'_+(t_0)$.

\end{remark}

\begin{remark}

The function $\rho_\kappa:\R\rightarrow \R$ plays a fundamental role in comparison theory for Riemannian manifolds. From a geometrical point of view, it controls the area element of the geodesic sphere in a space of constant curvature $\kappa$. In practice, its derivative $\rho_\kappa^{\prime}$ controls how radial geodesics spread from a point in a space of constant curvature $\kappa$.

\end{remark}

Moreover, for a Riemannian manifold without boundary, condition \cref{eq:concavity_condition} characterizes geodesics among all Lipschitz curves in the sense of the following proposition (see \cite{PetruninSemiconcave}).

\begin{proposition}

Assume that $\partial M =\varnothing$ and let $c:[a,b]\rightarrow M$ be a Lipschitz continuous curve in $M$ that is parameterized by arclength. Then $c$ is a geodesic segment of $M$ if and only if, for every $p\in M$ such that $0<d_{g}(p,c(t))<\alpha_\kappa$ for all $t\in [a,b]$, the continuous function 
$$t\mapsto \rho_\kappa(d_{g}(p,c(t)))$$ 
satisfies \cref{eq:concavity_condition}.

\end{proposition}

However, if $M$ has a non-empty boundary, there are Lipschitz curves in $M$ satisfying condition \cref{eq:concavity_condition} that are not geodesics. Here are two simple examples.

\begin{figure}
  \centering
  \includegraphics[width=6cm]{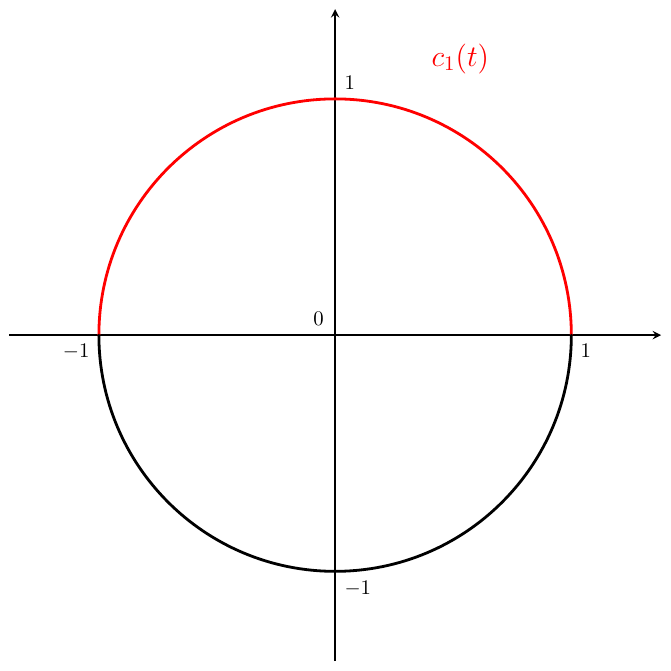}
  \includegraphics[width=6cm]{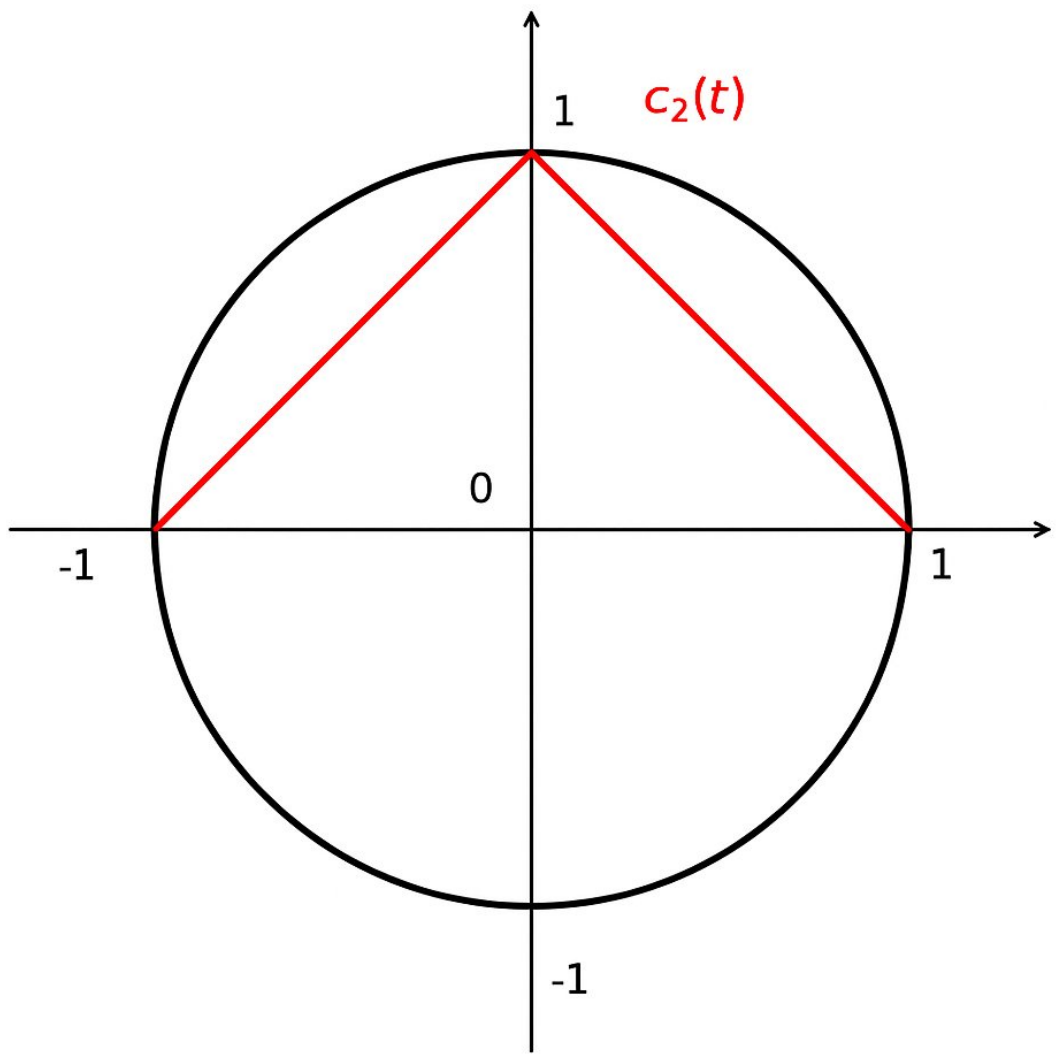}
  \caption{On the left, the smooth arc $c_1(t)$ of \Cref{ex:quasigeodesic_example} and, on the right, the piecewise straight curve $c_2(t)$ of \Cref{ex:billiard_example}.}\label{fig:quasigeodesic_examples}
\end{figure}

\begin{example}\label{ex:quasigeodesic_example}
Let $M = \overline{D}\subset \R^2$ be the closed unit disk in $\R^2$ endowed with the induced Euclidean metric, which is a non-negatively curved space. Consider the arclength-parameterized smooth curve
$$c_1(t) = (\cos(t),\sin(t)),\quad t\in[0,\pi],$$
that is depicted in \Cref{fig:quasigeodesic_examples}. A direct computation shows that for every $q = (x_0,y_0)\in \overline{D}\setminus c([0,\pi])$ the function
$$f_q(t) := \frac{|q-c_1(t)|^2}{2}$$ 
is smooth and satisfies $f_q^{\prime\prime}(t) = x_0\cos(t) + y_0\sin(t) \leq 1$ for each $t\in (0,\pi)$. Hence $c_1$ satisfies condition \cref{eq:concavity_condition} but it is not a geodesic since it is not locally length-minimizing for any time.
\end{example}

\begin{example}\label{ex:billiard_example}
Let $M = \overline{D}\subset \R^2$ be the closed unit disk in $\R^2$ endowed with the induced Euclidean metric as in \Cref{ex:quasigeodesic_example} and consider the arclength-parameterized piecewise smooth curve
$$c_2(t) = \begin{cases}
           \left(\dfrac{t}{\sqrt{2}},1+\dfrac{t}{\sqrt{2}}\right), & \mbox{if } t\in [-\sqrt{2},0],\\
           \left(\dfrac{t}{\sqrt{2}},1-\dfrac{t}{\sqrt{2}}\right), & \mbox{if } t\in [0,\sqrt{2}],
         \end{cases}$$
that is depicted in \Cref{fig:quasigeodesic_examples}. Then, a direct computation shows that $c_2$ satisfies condition \cref{eq:concavity_condition} as well. However, it is not a geodesic since it is not locally length minimizing at $t=0$.
\end{example}

\begin{definition}

Let $c:[a,b]\rightarrow M$ be a Lipschitz regular curve in $M$ which is parameterized by arclength. Then $c$ is called a \textit{$\kappa$-quasigeodesic curve of $M$} if for every $p\in M$ such that $0<d_g(p,c(t))<\alpha_\kappa$ for all $t\in [a,b]$, the continuous function
$$t\mapsto \rho_\kappa(d_g(p,c(t)))$$
satisfies \cref{eq:concavity_condition}.

\end{definition}

The following propositions (from \cite{PetruninSemiconcave}) collect some properties of $\kappa$-quasigeodesic curves in $M$ that will be used below.

\begin{proposition}\label{prop:geodesic_quasigeodesic}

Every $\kappa$-quasigeodesic $c:[a,b]\rightarrow M$ is a Lipschitz regular curve in $M$, i.e. there exist unique left and right derivatives at every $t\in [a,b]$. Moreover, if $M$ has empty boundary, a Lipschitz curve $c:[a,b]\rightarrow M$ is a $\kappa$-quasigeodesic if and only if it is a geodesic segment of $M$ which is parameterized by arclength.

\end{proposition}

\begin{proposition}\label{prop:quasigeodesic_polar}

Let $c: [a,b]\rightarrow M$ be a $\kappa$-quasigeodesic of $M$. Then, for every $t_0\in[a,b]$ such that $c(t_0)\in \partial M$, the right and left tangent vectors $c^+(t_0)$, $c^-(t_0)\in C_{c(t_0)}(M)$ are polar. \\
On the other hand, if $c_1:[a,b]\rightarrow M$ and $c_2:[b,d]\rightarrow M$ are $\kappa$-quasigeodesics such that $c_1(b)=c_2(b)\in \partial M$ and $c_1^-(b)$, $c_2^+(b)$ are polar, then the continuous curve $c:[a,d]\rightarrow M$ obtained by gluing $c_1$ and $c_2$ together is also a $\kappa$-quasigeodesic.

\end{proposition}

As discussed in \cite{PetruninSemiconcave}, all of the comparison theory for spaces of curvature bounded below by $\kappa\in \R$ as treated until now admits a natural extension to metric spaces satisfying suitable comparison conditions. These spaces are called Alexandrov spaces with curvature bounded below by $\kappa$, or $\CBB(\kappa)$-spaces\footnote{The class of $\CBB(\kappa)$-spaces contains all spaces of curvature bounded below by $\kappa$ defined as in \Cref{def:space_curvature_below}.}. We refer to \cite{Course} and the references therein for a detailed treatment.

An important class of non-trivial examples is given by doubles of smooth Riemannian manifolds with boundary along their boundary. More precisely, if $(M,g)$ is a space of curvature bounded below by $\kappa\in \R$, then its double $(D(M),d_{D(M)})$ is a $\CBB(\kappa)$ space. Therefore, the results contained in this subsection apply both to $M$ and to $D(M)$, even though the natural Riemannian metric on $D(M)$ is generally only Lipschitz across the boundary.

\subsection{Convergence of distance functions and quasigeodesic curves}

Let $(M,d_M)$ be a metric space. Also, let $\{d_k\}_{k\in \N}$ be a sequence of distance functions on $M$ (not necessarily equivalent to $d_M$). For the rest of the subsection, we will assume that $d_k\to d_M$ as $k\to \infty$ uniformly on compact subsets of $M\times M$.

We start by reporting the following important result, which is a \say{fixed-underlying set} version of a much more general statement about stability of minimizing geodesics under so-called pointed Gromov-Hausdorff convergence of metric spaces (we refer to \cite[Section 7.5]{Course} for all the details).

\begin{proposition}\label{prop:geodesic_convergence}
     Assume that $d_k$ converges to $d_M$ as $k\to \infty$ uniformly on compact subsets of $M\times M$. Also, for each $k\in \N$, let $\gamma_k:[a,b]\to M$ be a $d_k$-minimizing arclength-parameterized geodesic segment of $M$ and assume that the sequence $\{\gamma_k\}_{k\in \N}$ converges to some continuous curve $\gamma:[a,b]\to M$ as $k\to\infty$ uniformly in $[a,b]\subset \R$ with respect to $d_M$. Then the curve $\gamma$ is a $d_M$-minimizing geodesic segment of $M$ parameterized by arclength.
\end{proposition}

\begin{proof}
Let $K\subset M$ be a compact set such that $\gamma_k([a,b])\subset K$ and $\gamma([a,b])\subset K$ for all sufficiently large $k$, which exists by uniform convergence of the sequence. Then, for $k\in \N$, we set
$$\varepsilon_k:=\sup_{p,q\in K}|d_k(p,q)-d_M(p,q)|,$$
which converges to $0$ as $k\to\infty$ because the functions $d_k$ converge to $d_M$ uniformly on compact sets.

First, we claim that, for every fixed $s,t\in[a,b]$,
$$\lim_{k\to\infty}d_k(\gamma_k(s),\gamma_k(t))= d_M(\gamma(s),\gamma(t)).$$
Indeed, by triangle inequality, we have
\begin{equation*}
\begin{split}
\left|d_k(\gamma_k(s),\gamma_k(t))-d_M(\gamma(s),\gamma(t))\right| & \le\left|d_k(\gamma_k(s),\gamma_k(t))- d_M(\gamma_k(s),\gamma_k(t))\right| \\
& \quad+\left|d_M(\gamma_k(s),\gamma_k(t)) -d_M(\gamma(s),\gamma(t))\right|\\
& \le\varepsilon_k+d_M(\gamma_k(s),\gamma(s)) + d_M(\gamma_k(t),\gamma(t)).
\end{split}
\end{equation*}
where, in the last step, we used the inequality $|d_M(p,q)-d_M(p',q')|\le d_M(p,p') + d_M(q,q')$, for $p,p',q,q'\in M$, valid in every metric space. Then, we notice that the right-hand side goes to $0$, since $\gamma_k\to\gamma$ uniformly with respect to $d_M$.

Now, since each $\gamma_k$ is a $d_k$-length minimizing curve which is parameterized by arclength, for every $s,t\in[a,b]$, we have
$$ d_k(\gamma_k(s),\gamma_k(t)) =|t-s|$$
Passing to the uniform limit as $k\to\infty$, we get
$$d_M(\gamma(s),\gamma(t))= \lim_{k\to\infty}d_k(\gamma_k(s),\gamma_k(t)) = |t-s|.$$
Therefore $\gamma$ is a $d_M$-minimizing geodesic in $M$ which is parameterized by arclength.
\end{proof}

Unfortunately, a similar result does not hold for general local geodesic curves, as the minimizing radius may tend to $0$ in the uniform limit. However, it is known that if $M$ is a $\CBB(\kappa)$-space with respect to all the $d_k$'s (or a space of curvature bounded below by $\kappa\in \R$ if each $d_k$ is a Riemannian distance functions associated to a smooth Riemannian tensor), a similar result holds for $\kappa$-quasigeodesics in $M$. Indeed, if such a uniform lower bound $\kappa$ exists and $d_k$ converges to $d_M$ as $k\to \infty$ uniformly on compact sets, then $M$ is a $\CBB(\kappa)$-space with respect to $d_M$ as well (see \cite[Proposition 10.7.1]{Course}). Moreover, we have the following result, which is taken from \cite{PetruninSemiconcave}.

\begin{proposition}\label{prop:quasigeodesic_limit}
  Assume that $d_k$ converges to $d_M$ as $k\to \infty$ uniformly on compact subsets of $M\times M$. Also, for each $k\in \N$, let $\gamma_k:[a,b]\to M$ be a $\kappa$-quasigeodesic with respect to $d_k$ and assume that the sequence $\{\gamma_k\}_{k\in \N}$ converges to some continuous curve $\gamma:[a,b]\to M$ as $k\to\infty$ uniformly in $[a,b]\subset \R$ with respect to $d_M$. Then $\gamma$ is a $\kappa$-quasigeodesic with respect to $d_M$.
\end{proposition}

\section{Geodesics approximating billiards}

\subsection{Billiard tables, billiard trajectories and quasigeodesics}

The goal of this section is to introduce the notion of billiard tables and billiard trajectories, and to explore their dynamics and geometrical properties. In what follows, we assume that $g$ is a complete smooth Riemannian metric on $\R^n$. We start by recalling several basic definitions.

\begin{definition}

The Riemannian manifold with boundary $(K,g_K)$ is called a \textit{billiard table} if $K$ is the closure of an open connected subset of $\R^n$ such that $\partial K$ is of class $C^{2,1}$ and $g_K$ is the metric on $K$ induced by $g$.

\end{definition}

\begin{definition}

Let $(K,g_K)$ be a billiard table and $x_0\in \partial K$. Two unit vectors $u,v$ that belong to the half space $C_{x_0} (K)\subset T_{x_0} K$ satisfy the \textit{billiard reflection law} if $(g_K)_{x_0}(u+v, w) = 0$ for all $w\in T_{x_0}(\partial K)$.

\end{definition}

\begin{remark}\label{rmk:reflection_rule}

This definition of the billiard reflection law coincides with the classical one, in which the Riemannian acute angles formed by the unit vectors $u,v\in C_{x_0}(K)$ with $T_{x_0}(\partial K)$ are equal. Also, if $\nu(x_0)$ denotes the unique unit normal vector at $x_0\in \partial K$ pointing into $K$, then two unit vectors $u,v\in C_{x_0}(K)$ satisfy the billiard reflection law if and only if
$$u+v = s\ \nu(x_0),$$
where $s = 2(g_K)_{x_0}(u,\nu(x_0)) = 2(g_K)_{x_0}(v,\nu(x_0))\ge 0$.

\end{remark}

The following lemma shows that, for two unit vectors, satisfying the billiard reflection rule is equivalent to being polar in the sense of \Cref{def:polar_vectors}.

\begin{lemma}

Let $(K,g_K)$ be a billiard table and $x_0\in \partial K$. Then two unit vectors $u,v\in C_{x_0}(K)$ satisfy the billiard reflection law if and only if they are polar.

\end{lemma}

\begin{proof}
  
Let $\nu(x_0)\in T_{x_0}K$ be the unique unit normal vector to $\partial K$ at $x_0$ pointing into $K$. Then the tangent space $T_{x_0} K$ splits into the orthogonal direct sum 
$$T_{x_0} K = T_{x_0}(\partial K)\oplus \mbox{span}_{\R}\{\nu(x_0)\}.$$ 
Also, we recall that every vector $w\in C_{x_0}(K)$ has the form $w = \tilde w + r\ \nu(x_0)$ with $\tilde{w}\in T_{x_0}(\partial K)$ and $r\ge0$.

\begin{itemize}

\item[$(\Rightarrow)$] Assume that two unit vectors $u,v\in C_{x_0}(K)$ satisfy the billiard reflection law, i.e. $u + v = s \ \nu(x_0)$ for some $s\ge 0$ (see \Cref{rmk:reflection_rule}). Then, for every $w = \tilde{w} + r\ \nu(x_0)$, we compute
    $$(g_K)_{x_0}(u+v,w) = (g_K)_{x_0}(s\ \nu(x_0),\tilde{w} + r\ \nu(x_0)) = sr\ (g_K)_{x_0}(\nu(x_0),\nu(x_0)) = sr\ge0.$$
    Hence $u$ and $v$ are polar.
\end{itemize}

\begin{itemize}

\item[$(\Leftarrow)$] Assume that two unit vectors $u,v\in C_{x_0}(K)$ are polar. Let $w\in T_{x_0}(\partial K)$ and recall that $-w\in T_{x_0}(\partial K)$. Then, by definition of polar vectors, we have
    $$(g_K)_{x_0}(u+v, w)\ge 0 \quad \text{and}\quad (g_K)_{x_0}(u+v,-w) \ge 0,$$
    which imply $(g_K)_{x_0}(u+v,w) = 0$. Therefore $u$ and $v$ satisfy the billiard reflection law.
\end{itemize}

\end{proof}

\begin{definition}

Let $K$ be a billiard table and $c:[a,b]\rightarrow K$ be a Lipschitz regular curve. Then $c: [a,b]\rightarrow K$ is called a \textit{billiard trajectory} if it is parameterized by arclength and it is locally length-minimizing at all but a finite set of \emph{bounce times}
$$\mathcal{T}:= \{\tau\in [a,b]: \ c(\tau)\in \partial K\}\subset [a,b]$$ 
such that for every $\tau\in \mathcal{T}$, the right and left tangent vectors $c^+(\tau)$, $c^-(\tau)$ are polar.

\end{definition}

In this paper we restrict our attention to nonsingular billiard trajectories, namely trajectories with only finitely many transverse reflections. For tables with boundary of class $C^{2,1}$, transverse reflections cannot accumulate in finite time, as proved in \cite{Halpern}. Also, we refer to \cite{Kourganoff} for the case in which trajectories can glide along the boundary of the table and their hyperbolic properties.
 
This construction of the billiard dynamics has several important implications, one of them being the uniqueness of the billiard flow, which is proved by the following proposition. For further details on the topic see, for
instance, \cite{tabac} or \cite[Section 3]{khtab} or
\cite{FierobeKaloshinSorrentino}.

\begin{proposition}\label{prop:billiard_uniqueness}
  Let $K$ be a billiard table and let $c_1,c_2:[a,b]\to K$ be billiard trajectories. Assume that there is $(d,e)\subset [a,b]$ such that $c_1(t) = c_2(t)$ for $t\in (d,e)$, then $c_1\equiv c_2$ in $[a,b]$.
\end{proposition}

\begin{proof}
Let $\mathcal T_i\subset [a,b]$ denote the finite set of bounce times of $c_i$, for $i=1,2$. Since $\mathcal T_1\cup\mathcal T_2$ is finite and $(d,e)$ is open, we may choose $t_0\in (d,e)\setminus(\mathcal T_1\cup\mathcal T_2)$. Then both $c_1$ and $c_2$ are ordinary smooth geodesics in a neighborhood of
$t_0$. Since they agree on the open interval $(d,e)$, they have the same initial data at $t_0$:
$$c_1(t_0)=c_2(t_0), \qquad \dot c_1(t_0)=\dot c_2(t_0).$$
By uniqueness for the geodesic equation, the two curves agree on the maximal
open interval, containing $t_0$, on which neither curve reflects.

Next, we show that equality propagates across bounce times. Suppose that $c_1$ and $c_2$ agree on an interval of the form $(\tau-\varepsilon,\tau)$, where $\tau$ is a bounce time for one of the two curves. Then $c_1(\tau)=c_2(\tau)=:x_0\in \partial K$ and the left tangent vectors agree, namely
$$c_1^-(\tau)=c_2^-(\tau)=:v^-.$$
Let $\nu(x_0)$ be the unit normal to $\partial K$ at $x_0$ pointing into $K$. The billiard reflection law gives
$$c_i^+(\tau) + v^-= 2\,g_{x_0}(v^-,\nu(x_0))\,\nu(x_0), \quad i=1,2.$$
Hence $\dot c_1(\tau^+)=\dot c_2(\tau^+)$. Therefore the two outgoing geodesic arcs have the same initial position and the same initial velocity. By uniqueness for the geodesic equation, they agree on a right-neighborhood of $\tau$. Notice that the same argument shows that if the two curves agree on $(\tau,\tau+\varepsilon)$, then their right velocities agree at $\tau$, and the reflection law uniquely determines the incoming velocity. Hence the two incoming geodesic arcs agree on a left-neighborhood of $\tau$.

Since each billiard trajectory has only finitely many reflection times, this same argument can be iterated across all reflection times between $t_0$ and the endpoints $a,b$. Thus
$$c_1(t)=c_2(t)$$
for every $t\in [a,b]$.
\end{proof}

It is well known that billiard dynamics in $K$ is related to the geodesic flow of its double $D(K)$. In \cite{tabac}, the author establishes a one-to-one correspondence between billiard trajectories of $K$ and geodesic segments of the double $D(K)$ which cross $\partial K$ transversally.

The second part of this subsection is devoted to establishing a similar correspondence involving quasigeodesic curves introduced in section 2.

\begin{theorem}\label{thm:quasigeodesic_characterization}

Let $(K,g)$ be a billiard table and assume that it is a space of curvature bounded below by some $\kappa\in \R$ so that $D(K)$ is a $\CBB(\kappa)$-space. Let $\tilde c:[a,b]\to D(K)$ be a $\kappa$-quasigeodesic of $D(K)$, and assume that $c:=\pi_K\circ \tilde c:[a,b]\to K$ intersects $\partial K$ only transversally. Then $c$ is a billiard trajectory of $K$.

\end{theorem}

\begin{proof}

Let $\tilde c:[a,b]\to D(K)$ be a $\kappa$-quasigeodesic of $(D(K),d_{D(K)})$ and assume that $c:=\pi_K\circ \tilde c:[a,b]\to K$ intersects $\partial K$ only transversally. Then, $\tilde c$ can also cross $\partial K$ only transversally (see \Cref{prop:double_projection}). In particular, this implies that $\tilde c$ can cross $\partial K$ at most finitely many times, and let us denote $\mathcal{T}\subset [a,b]$ this set of times.

First, we observe that, since $\tilde c$ is a $\kappa$-quasigeodesic, then $c$ is a Lipschitz regular curve in $K$ parameterized by arclength thanks to \Cref{prop:double_projection}.

Next, we claim that for each $t_0\in \mathcal{T}$ the unit vectors $c^+(t_0)$ and $c^-(t_0)$ are polar. Indeed, since $\tilde c$ is a $\kappa$-quasigeodesic, for each $t_0\in \mathcal{T}$ the unit vectors $\tilde c^+(t_0)$ and $\tilde c^-(t_0)$ are polar (see \Cref{prop:quasigeodesic_polar}). Moreover, \Cref{prop:double_projection} says that the tangent map $d_q\pi_K$ transforms polar vectors in $T_qD(K)$ to polar vectors in $C_{\pi_K(q)}K$.

To conclude, it remains to show that $c$ is locally length-minimizing at every $t_*\in [a,b]\setminus \mathcal{T}$. Notice that if $t_*\in [a,b]\setminus\mathcal{T}$ there exists $\varepsilon>0$ such that $(t_*-\varepsilon,t_*+\varepsilon)\cap \mathcal{T} =\emptyset$. Indeed, if it were not the case, there would be a sequence of times $\{t_i\}_{i\in \N}\subset\mathcal{T}$ such that $t_i\rightarrow t_*$ as $i\rightarrow \infty$. But, since $c$ is a continuous function and $\partial K$ is a closed set, we would have $c(t_*) = \lim\limits_{i \rightarrow \infty}c(t_i)\in \partial K$, contradicting the assumption $t_*\notin \mathcal{T}$. Therefore, we fix $t_*\in (a,b)\setminus\mathcal{T}$ and $\varepsilon>0$ such that $(t_*-\varepsilon,t_*+\varepsilon)\cap \mathcal{T} =\emptyset$ and we observe that the curve 
$$\hat{c}: = \tilde c\big|_{(t_*-\varepsilon,t_*+\varepsilon)} : (t_*-\varepsilon,t_*+\varepsilon)\rightarrow K$$
is a $\kappa$-quasigeodesic that is entirely contained in the interior of one the sheets of $D(K)$, say $\Int(K_+)$. Now, since $\Int(K_+)$ is a Riemannian manifold without boundary, \Cref{prop:geodesic_quasigeodesic} implies that $\hat{c}$ is a geodesic of $\Int(K_+)$. Hence the curve $c|_{(t_*-\varepsilon,t_*+\varepsilon)} = \pi_K\circ \hat c$ is a geodesic of $K$. Thus, $c$ is a billiard trajectory by definition.

\end{proof}

\subsection{Families of folds}

Let $K\subset \R^n$ be a billiard table as defined in the previous section and identify it with its image in $\R^{n+1}$ through the embedding $\R^n\hookrightarrow \R^{n+1}$, $x\mapsto (x,0)$. 
In what follows, let $\overline{g}$ be a smooth complete Riemannian metric on $\R^{n+1}$ and consider the hyperplane 
$$H:=\left\{(x,z)\in \R^n\times \R:\ z=0\right\}.$$

The goal of this section is to define a sufficiently general notion of folds over $K$ which should resemble the family of ellipsoids given in the introduction. The prototype of such notion is given by the following local construction.

Let $p_0\in \partial K$ and assume that $\g$ is the Euclidean metric in $\R^{n+1}$. Let $U\subset \R^{n+1}$ be an open bounded neighborhood of $p_0$ and choose $f\in C^{2,1}(K\cap\overline{U})$ with a regular value at $0$ and such that
$$K\cap U =\{(x,0)\in H: f(x)\ge 0\},\quad \partial K \cap U = f^{-1}(0).$$
Now, we construct the family of Riemannian hypersurfaces $\left\{\left(\M,g_\lambda\right)\right\}_{\lambda\in (0,1)}$, where
\begin{equation}\label{eq:family_folds_local}
\M = \left\{(x,x_{n+1})\in \R^{n+1}:\ x_{n+1}^2 = \lambda^2f(x),\ x\in K\cap U\right\}\subset \R^{n+1}
\end{equation}
and $g_\lambda$ is the metric on $\M$ induced by $\overline{g}$.

\begin{figure}
  \begin{minipage}{0.45\textwidth}
    \includegraphics[width=7.3cm]{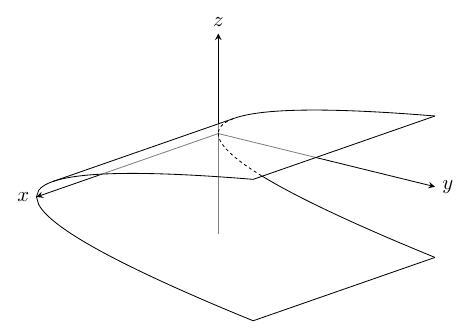}
  \end{minipage}
  \hfill
  \begin{minipage}{0.45\textwidth}
    \includegraphics[width=7.3cm]{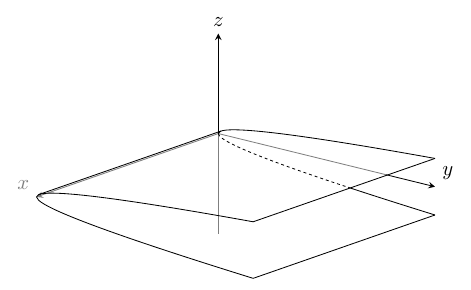}
  \end{minipage}
  \centering
  \begin{minipage}{0.45\textwidth}
    \includegraphics[width=7.3cm]{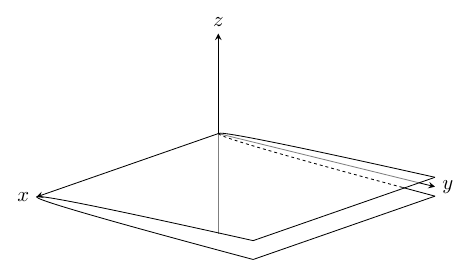}
  \end{minipage}
  \caption{A portion of the parabolic cylinders $\{z^2 = \lambda^2 y\}\subset \R^3$ for $\lambda = 0.5,0.2$ and $0.06$ respectively.}\label{fig:parabolic_cylinders}
\end{figure}

\begin{example}
Let $K$ be a half-space of $H$, e.g. $K = \{x_2\ge 0, \ x_{n+1}=0\}$. Fix $p_0\in \partial K$ and an open bounded neighborhood $U\subset H$ of $p_0$. Then, each $\M$ is a portion of an $n$-dimensional parabolic cylinder. \Cref{fig:parabolic_cylinders} illustrates the case $n=2$, showing the family of folds over a neighborhood of $p_0 = (0,1,0)$, for several values of $\lambda$.
\end{example}

\begin{example}
Let $n=2$ and $K = \{(x_1,x_2,x_3)\in \R^3:\ 1-x_1^2-x_2^2 \ge 0, \ x_3=0\}$ be the closed unit disk in $H\subset \R^3$. Then, if $p_0\in\partial K$ and $U$ is an open bounded subset of $H$ containing $K$, the family of surfaces in \cref{eq:family_folds_local} agrees with the family of ellipsoids we defined in the introduction:
$$M_\lambda = \Big\{(x_1,x_2,x_3)\in \R^3: \ x_3^2 = \lambda^2(1-x_1^2-x_2^2)\Big\} = \left\{(x_1,x_2,x_3)\in \R^3: \ x_1^2 + x_2^2+ \frac{x_3^2}{\lambda^2} = 1\right\}.$$
\end{example}

If $\g$ is not the Euclidean metric, we now present a similar local construction modeled on \eqref{eq:family_folds_local}. Let $\nu$ be a $\g$-unit normal vector field along $H=\{x_{n+1} = 0\}$ and define $F:H\times (-\rho,\rho)\to \R^{n+1}$ as 
$$F(x,t) = \exp_x^{\g}(t\nu_x),$$
where $\exp^{\g}$ is the standard exponential map associated to the Riemannian metric $\g$. Notice that, if $\rho>0$ is small enough, $F$ is a diffeomorphism onto a tubular neighborhood of $H$. Here, we define the collection $\{(M_\lambda,g_\lambda)\}_{\lambda\in (0,1)}$, as
$$M_\lambda = \{F(x,t)\in \R^{n+1}:\ t^2 = \lambda^2f(x), \ x\in K\cap U,\ t\in (-\rho,\rho)\}$$
and $g_\lambda$ is the metric on $M_\lambda$ induced by $\g$.

Looking at the above constructions, it is clear that our definition of family of folds over a table $K$ should include hypersurfaces made of two sheets which are joined at the boundary $\partial K$, and which flatten onto at least a portion of $K$ including some of $\partial K$ as $\lambda\to 0^+$. Indeed, the families presented above have several useful properties which are embodied in the following general definition.  

\begin{definition}\label{def:family_folds}

Let $p_0\in \partial K$. A \emph{family of folds over $p_0$} is a collection of embedded Riemannian hypersurfaces $\{(M_\lambda,g_\lambda)\}_{\lambda\in(0,1)}$ of $\R^{n+1}$ such that

\begin{enumerate}
  \item $g_\lambda$ is the Riemannian metric induced on $M_\lambda$ by the complete ambient metric $\g$;
  \item for every $\lambda\in (0,1)$ we have $M_\lambda\cap H \subseteq \partial K$;
  \item there is a compact neighborhood $C$ of $p_0$ in the double $D(K)$ and a map $\Phi_\lambda:\, C\to M_\lambda$ such that the following properties hold.
      \begin{itemize}
        \item $\Phi_\lambda$ is a topological embedding of topological manifolds with boundary;
        \item $\Phi_\lambda\big|_{\partial K\cap C} = \id_{\partial K\cap C}$;\vskip0.2cm
        \item $\lim\limits_{\lambda\to 0^+}\; d_{\g} (\Phi_\lambda(q), \pi(q)) = 0$ uniformly in $q\in C$.
        \item there is $T_0>0$ such that, for each $\lambda\in (0,1)$, every arclength-parameterized geodesic segment in $M_\lambda$ starting from $p_0$ can be extended to $[-T_0,T_0]$ and its image is contained in the interior of $\Phi_\lambda(C)$;
        \item $\lim\limits_{\lambda\to 0^+}\;\left|d_{g_\lambda}(\Phi_\lambda(q_1), \Phi_\lambda(q_2)) - d_{D(K)}(q_1,q_2)\right| = 0$ uniformly in $q_1,q_2\in C$.
      \end{itemize}
\end{enumerate}

\end{definition}

\begin{remark}
  Notice that, if $\{(M_\lambda,g_\lambda)\}_{\lambda\in(0,1)}$ is a family of folds over $p_0\in \partial K$ and $C\subset D(K)$ is the compact neighborhood of $p_0$ in which $\Phi_\lambda$ is defined, the distance functions $d_\lambda:C\times C\to \R$, defined by
  $$d_\lambda(q_1,q_2) = d_{g_\lambda}(\Phi_\lambda(q_1),\Phi_\lambda(q_2)),$$
  converge to $d_{D(K)}$ as $\lambda\to 0^+$ uniformly in $C\times C$.
\end{remark}

We postpone to Appendix A the proof of the fact that the local construction \cref{eq:family_folds_local} is indeed a family of folds over $p_0$ in the sense of \Cref{def:family_folds}, as it requires some time.

\subsection{Proof of the main theorem}

This subsection is devoted to proving \Cref{thm:main_theorem}. Let $K\subset H\subset \R^{n+1}$ be a billiard table, $p_0\in \partial K$ and let $(M_\lambda,g_\lambda)$ be a family of folds over $p_0$ in the sense of \Cref{def:family_folds}. Before we state and prove the main theorem, let us define the notation that will be valid until the end of this subsection.

We fix a decreasing sequence $\lambda_k\to 0^+$ in $(0,1)$ and we denote
$M_k:= M_{\lambda_k}$, $g_k:= g_{\lambda_k}$. Let $C\subset D(K)$ be a compact neighborhood of $p_0$, $\Phi_k: C\to M_k$ and $T_0>0$ be as in \Cref{def:family_folds}. Also, we denote by $d_k:C\times C\to \R$ the distance function defined by 
$$d_k(q_1,q_2)= d_{g_k}(\Phi_k(q_1),\Phi_k(q_2)).$$
In the following, we formulate two extra crucial assumptions.

\begin{assumption}\label{ass:assumption1}
There exists a $\kappa\in \R$ such that, for every $k\in \N$, there is a neighborhood of $\Phi_k(C)$ in $M_k$ which is a space of curvature bounded below by $\kappa$.
\end{assumption}

\begin{assumption}\label{ass:assumption2}
There is $\tilde T>0$ such that, for every $k\in \N$, each geodesic segment $\gamma:[-\tilde T,\tilde T]\to M_k$ with $\gamma(0) = p_0$ is $g_k$-length minimizing.
\end{assumption}

\begin{theorem}(= \Cref{thm:compactness_intro})\label{thm:main_theorem_sec3}
Assume that either \Cref{ass:assumption1} or \Cref{ass:assumption2} holds. Then there exists $T>0$ such that every sequence of arclength-parameterized geodesic segments
$$\left\{\gamma_k:[-T,T]\rightarrow M_k\right\}_{k\in \N},$$
that satisfies $\gamma_k(0)=p_0$ for all $k\in\N$, admits a subsequence $\{\gamma_m\}_{m\in \N}$ that converges to a continuous curve $\gamma:[-T,T]\rightarrow K$ uniformly in $[-T,T]$ as $m\rightarrow \infty$. Moreover, if $\gamma$ intersects $\partial K$ transversally, then it is a billiard trajectory of $K$.

\end{theorem}

\begin{proof}

We take $T>0$ such that $T<T_0$ and we let $\{\gamma_k\}_{k\in \N}$ be a sequence of geodesic segments $\gamma_k:[-T,T]\rightarrow \Int_{M_k}(\Phi_k(C))$, parameterized by arclength, such that $\gamma_k(0)=p_0$ for every $k\in \N$.

We consider the sequence of curves $\{\tilde \gamma_k\}_{k\in\N}$, where each
$$\tilde\gamma_k= \Phi_k^{-1}\circ \gamma_k : [-T,T]\to \Int(C)$$
is an arclength-parameterized geodesic segment of $C$ (with respect to $d_k$) such that $\tilde\gamma_k(0) = p_0$, and we claim that $\{\tilde\gamma_k\}_{k\in \N}$ satisfies the assumptions of the Ascoli-Arzelà theorem for the metric space $(C, d_{D(K)})$.

\begin{itemize}
    \item Since $C$ is compact, the sequence of curves is uniformly bounded.
    \item We claim that the sequence is uniformly equicontinuous. Set $\varepsilon_k:=\sup_{q_1,q_2\in C}|d_k(q_1,q_2)-d_{D(K)}(q_1,q_2)|$ and notice that, for each $t,s\in [-T,T]$,
        $$d_{D(K)}(\tilde\gamma_k(t),\tilde\gamma_k(s)) \le d_k(\tilde\gamma_k(t),\tilde\gamma_k(s)) + \varepsilon_k\le |t-s| + \varepsilon_k.$$
        Then, for each $\varepsilon>0$ there is $k_0\in \N$ such that $\varepsilon_k<\varepsilon/2$ whenever $k>k_0$. Thus, if $k>k_0$ and $|t-s|< \varepsilon/2$, we have
        $$d_{D(K)}(\tilde\gamma_k(t),\tilde\gamma_k(s)) < \varepsilon.$$
        Moreover, since for every $k\in \N$ the curve $\gamma_k$ is uniformly continuous, we have that for each $k\in \N$ and $\varepsilon>0$, there is $\delta_k>0$ such that if $|t-s|<\delta_k$, it holds $d_{D(K)}(\tilde\gamma_k(t),\tilde\gamma_k(s))<\varepsilon/2$. 
         
        Thus if we fix $\varepsilon>0$ and let $k_0\in \N$ as above, choose $\delta := \min(\delta_0,\ldots,\delta_{k_0},\varepsilon/2)$, and notice that, if $|t-s|<\delta$, we have
        $$d_{D(K)}(\tilde\gamma_k(t),\tilde\gamma_k(s))<\varepsilon$$
        for all $k\in \N$. 
        
\end{itemize}

As a consequence, the Ascoli-Arzelà theorem implies that there exists a subsequence $\{\tilde\gamma_m\}_{m\in \N}$ of $\{\tilde\gamma_k\}_{k\in \N}$ that converges to a certain continuous curve $\tilde\gamma : [-T,T]\rightarrow \Int(C)$ uniformly in $[-T,T]$ as $m\rightarrow \infty$. 

Next, we take $\gamma:=\pi\circ \tilde\gamma:[-T,T]\to K$ and we claim that $\gamma_m$ converges to $\gamma$ as $m\to\infty$ uniformly in $[-T,T]$ with respect to the ambient distance $d_{\g}$. Here we observe that, for every $t\in [-T,T]$,
\begin{equation}
\begin{split}
d_{\g}(\gamma_m(t), \pi(\tilde\gamma(t))) &\le\, d_{\g}(\gamma_m(t), \pi(\tilde\gamma_m(t)) + d_{\g}(\pi(\tilde\gamma_m(t)), \pi(\tilde\gamma(t))) \\ 
&= \, d_{\g}(\Phi_m(\tilde\gamma_m(t)), \pi(\tilde\gamma_m(t))) + d_{\g}(\pi(\tilde\gamma_m(t)), \pi(\tilde\gamma(t))),
\end{split}
\end{equation}
as $\Phi_m\circ \tilde\gamma_m = \gamma_m$ by definition. Notice that the first term tends to $0$ as $m\to \infty$, indeed
$$d_{\g}(\Phi_m(\tilde\gamma_m(t)), \pi(\tilde\gamma_m(t))\le \sup_{q\in C}\,d_{\g}(\Phi_m(q), \pi(q))\to 0,\quad \mbox{as}\ m\to \infty.$$
Moreover, we estimate the second term as
$$d_{\g}(\pi(\tilde\gamma_m(t)), \pi(\tilde\gamma(t)))\le d_{g_K}(\pi(\tilde\gamma_m(t)), \pi(\tilde\gamma(t)))\le d_{D(K)}(\tilde\gamma_m(t), \tilde\gamma(t))\to 0,$$
as $m\to \infty$, thanks to the uniform convergence of $\tilde \gamma_m$ to $\tilde\gamma$ in $D(K)$ and to the fact that $\pi:D(K)\to K$ is a $1$-Lipschitz map.

Now, we notice that if \Cref{ass:assumption1} holds, then each distance function $d_k$ turns $\Int(C)$ into a space of curvature bounded below by some $\kappa\in \R$ and so does $d_{D(K)}$ because $d_k\to d_{D(K)}$ as $k\to \infty$ uniformly in $C$ (see discussion in section 2.4). In particular, every $\gamma_k$ is a $\kappa$-quasigeodesic of $(\Int(C),d_k)$. Then, \Cref{prop:quasigeodesic_limit} implies that $\tilde\gamma$ is a $\kappa$-quasigeodesic of $(\Int(C),d_{D(K)})$.

Alternatively, if \Cref{ass:assumption2} holds and up to choosing a smaller $T$ so that $T<\tilde T$, each $\tilde\gamma_m$ is length-minimizing with respect to $d_k$. As a consequence, \Cref{prop:geodesic_convergence} implies that $\tilde\gamma$ is a $d_{D(K)}$-minimizing geodesic segment of $\Int(C)$.

In either case, since $\gamma = \pi_K\circ \tilde \gamma$, then if $\gamma$ intersects $\partial K$ only transversally, \Cref{thm:quasigeodesic_characterization} (and the preceding discussion) implies that $\gamma$ is a billiard trajectory of $\pi_K(C)\subset K$ and hence of $K$. 

\end{proof}

\begin{corollary}[= \Cref{thm:main_theorem}]\label{cor:main_corollary_sec3}

Let $T>0$ as in \Cref{thm:main_theorem_sec3} and assume that either \Cref{ass:assumption1} or \Cref{ass:assumption2} holds. For every billiard trajectory $\gamma:[-T,T]\to K$ reflecting at $p_0:=\gamma(0)\in \partial K$, there is a sequence of arclength-parameterized geodesic segments $\gamma_k:[-T,T]\rightarrow M_k$ with $\gamma_k(0)=p_0$ such that $\gamma_k$ converges to $\gamma$ as $k\rightarrow \infty$ uniformly in $[-T,T]$ with respect to the ambient distance $d_{\g}$.

\end{corollary}

\begin{proof}
  
Let us assume that $p_0=(x_0,0)\in \partial K$ is an isolated reflection of the billiard trajectory $\gamma$. Since $\gamma$ reflects transversally at $p_0$, we can pick $p_*:=\gamma(t_*)$ for a $t_*\in[0,T]$ sufficiently small such that $\gamma|_{[0,t_*]}$ is the unique minimizing geodesic in $\Int(K)$ joining $p_0$ and $p_*$ and both lifts of $p_*$ to $D(K)$ are in the interior of $C\subset D(K)$. Then, if $\tilde p_*$ is one fixed lift of $p_*$, we consider the points $p_k:= \Phi_k(\tilde p_*)\in M_k$ and, by definition of family of folds, we observe that
$$\lim_{k\to \infty}\, d_{\g}(p_k,p_*) = \lim_{k\to \infty}\, d_{\g}(\Phi_k(\tilde p_*),\pi(\tilde p_*))= 0$$
and
$$d_{g_K}(p_0,p_*) = d_{D(K)}(p_0,\tilde p_*) = \lim_{k\to \infty}\, d_{g_k}(p_0,p_k).$$
In particular, considering only $k>k_0$ for some fixed $k_0$ large enough, it is not restrictive to assume $t_k:=d_{g_k}(p_0,p_k)< T$, as $t_*< T$.

Now, let $c_k:[0,t_k]\rightarrow M_k$ be a $d_{g_k}$-minimizing arclength-parameterized curve in $\Phi_k(C)$ connecting $p_0$ and $p_k$, which always exists thanks to the direct method of calculus of variations since $(\Phi_k(C),d_{g_k})$ is compact length space, and we claim that it is indeed a Riemannian geodesic of $\Int_{M_k}(\Phi_k(C))$. Since $p_0\in\Int_{M_k}(\Phi_k(C))$, there is a small time $0<\hat t\le t_k$ for which $c_k\big|_{[0,\hat t]}$ is a Riemannian geodesic of $\Int_{M_k}(\Phi_k(C))$. Since $t_k\le T<T_0$, \Cref{def:family_folds} says that $c_k\big|_{[0,\hat t]}$ extends to a geodesic segment $\hat c_k:[0,t_k]\to \Int_{M_k}(\Phi_k(C))$ and the uniqueness of the geodesic initial-value problem yields $c_k \equiv \hat c_k$. Moreover, we can further extend $c_k$ to a geodesic segment $\gamma_k:[-T,T]\rightarrow \Int_{M_k}(\Phi_k(C))$. 

Here, \Cref{thm:main_theorem_sec3} implies that, up to extracting a subsequence and identifying $\gamma_k$ with it, $\gamma_k$ converges to a billiard trajectory $\hat{\gamma}:[-T,T]\to K$ uniformly in $[-T,T]$. In particular, since $\hat\gamma$ is arclength-parameterized, we have 
$$\mbox{Length}_K(\hat\gamma|_{[0,t_*]})= t_* = \mbox{Length}_K(\gamma|_{[0,t_*]}).$$
Also, we can check that $\hat\gamma(t_*) = p_*$:
$$d_{\g}(\hat\gamma(t_*),p_*)\le \lim_{k\to \infty}(d_{\g}(\hat\gamma(t_*),\hat\gamma(t_k)) + d_{\g}(\hat\gamma(t_k),\gamma_k(t_k)) + d_{\g}(p_k,p_*)) = 0.$$
Thus, since $\hat\gamma(t_*) = p_* = \gamma(t_*)$, the two curves must coincide in $[0,t_*]$ since $\gamma$ was assumed to be the unique short geodesic joining $p_0$ and $p_*$.

Therefore, since two billiard trajectories coincide in a non-empty open interval, they must be the same by \Cref{prop:billiard_uniqueness}.

\end{proof}

\section{Examples of admissible Euclidean billiard tables}

In this section we present types of billiard tables $K$ to which \Cref{thm:main_theorem} is applicable, i.e. either \Cref{ass:assumption1} or \Cref{ass:assumption2} is satisfied. In what follows, we set $\g$ to be the Euclidean metric on $\R^{n+1}$, $H=\{(x,x_{n+1})\in\R^{n+1}:\ x_{n+1}=0\}$ and we let $K\subset H$ be a billiard table. 

Recall that for any fixed $p_0\in \partial K$, there is an open bounded neighborhood $U\subset \R^{n+1}$ of $p_0$ and a function $f\in C^{2,1}(\overline{U}\cap K)$ with a regular value at $0$ such that
\begin{equation}\label{eq:table}
  K\cap U = \{(x,0)\in H:\ f(x)\ge 0\}\quad \text{and}\quad \partial K\cap U = f^{-1}(0).
\end{equation}
Also, we will denote by $\nabla f(x)$ and $D^2f(x)$ the Euclidean gradient and Hessian matrix of $f$ respectively. 

In the following propositions we will always refer to the specific family of folds $\{\M,g_\lambda\}_{\lambda\in (0,1)}$ constructed in \cref{eq:family_folds_local}, where
$$\M = \left\{(x,x_{n+1})\in \R^{n+1}:\ x_{n+1}^2 = \lambda^2f(x),\ x\in K\cap U\right\}\subset \R^{n+1}$$
and $g_\lambda$ is the metric on $\M$ induced by $\overline{g}$.

\begin{proposition}\label{prop:convex_billiards}

Let $K\subset H$ be a convex billiard table. Then, for every $p_0\in \partial K$, there exists an open bounded neighborhood $U\subset \R^{n+1}$ of $p_0$ and $f\in C^{2,1}(\overline{U}\cap K)$ satisfying \cref{eq:table} such that the sectional curvatures of the fold $(\M,g_\lambda)$ are nonnegative at every point for each $\lambda\in (0,1)$, i.e. \Cref{ass:assumption1} is satisfied with $\kappa = 0$.

\end{proposition}

\begin{proof}
  
Let $p_0\in \partial K$ be a boundary point. Since $K$ is a convex billiard table, there exists an open bounded neighborhood $U\subset \R^{n+1}$ of $p_0$ and $f\in C^{2,1}(\overline{U}\cap K)$ satisfying \cref{eq:table} such that $D^2f(x)$ is negative semi-definite for every $(x,0)\in K\cap U$. 

If $\{\M\}_{\lambda\in (0,1)}$ denotes the family of folds \eqref{eq:family_folds_local}, we notice that $\M$ is the zero-level set of the function
$$F_\lambda(x,x_{n+1}) = x_{n+1}^2 -\lambda^2 f(x),$$
which has a regular value at $0$. Then, we know that, for every $q=(x,x_{n+1})\in \M$, the scalar second fundamental form of $\M$ at $q$ (associated to the unit normal field given by the gradient of $F_\lambda$) reads
$$h_q(v,w)= \frac{v^TD^2F_\lambda(q)\ w}{|\nabla F_\lambda(q)|}, \quad \text{for}\ v,w\in T_q\M,$$
where $\nabla F_\lambda(q)$ and $D^2F_\lambda(q)$ denote the Euclidean gradient and Hessian matrix of $F_\lambda$ at $q$ respectively.
We further notice that
$$D^2F_\lambda(x,x_{n+1}) = \begin{pmatrix}-\lambda^2 \ D^2f(x) & 0\\ 0 & 2\end{pmatrix},$$
which is positive semi-definite since $D^2f(x)$ is negative semi-definite. Thus, the bilinear form $h_q$ is positive semi-definite.

Now, we recall that for every pair of linearly independent vectors $v,w\in T_q\M$, the Gauss equation for a Riemannian hypersurface (see \cite[Theorem 8.13]{LeeRiemannian}) yields that the sectional curvature of $\M$ at $q$ associated to the plane spanned by $v,w$ is
$$\mbox{sec}^q_\lambda(v,w) = \frac{h_q(v,v)h_q(w,w)- h_q(v,w)^2}{\overline{g}_q(v,v)\overline{g}_q(w,w)- \overline{g}_q(v,w)^2}.$$
Since $h_q$ is positive semi-definite and $\overline{g}_q$ is positive definite, it is well known that for all linearly independent $v,w\in T_q\M$ it holds
$$h_q(v,v)h_q(w,w)- h_q(v,w)^2\ge 0 \quad \text{and}\quad \overline{g}_q(v,v)\overline{g}_q(w,w)- \overline{g}_q(v,w)^2>0.$$ 
This implies that $\mbox{sec}_\lambda^q(v,w)\ge 0$ for every $\lambda\in (0,1)$.

\end{proof}

Therefore, \Cref{prop:convex_billiards} and \Cref{cor:main_corollary_sec3} prove \Cref{thm:convex_billiard_intro} in the Introduction. However, \Cref{ass:assumption1} is not satisfied for certain concave tables.

\begin{example}

Let $n=2$ and let 
$$K=\{(x_1,x_2,0)\in \R^3:\ x_1^2-x_2\ge0\}$$ 
be the concave portion of the $x_1x_2$-plane whose boundary is the parabola $\{x_2=x_1^2\}$. Let $p_0=(0,0,0)\in \partial K$ and let $U\subset \R^3$ be an open, bounded neighborhood of $p_0$. Then, a direct computation shows that the Gaussian curvature of the surface $\M$ over $K\cap U$ at $p_0$ is
$$G_{\M}(p_0) = -\frac{4}{\lambda^2},$$
which is not bounded below uniformly in $\lambda$.

However, for each $\lambda\in (0,1)$, $\M$ is an open subset of a closed simply connected manifold of non-positive curvature. Therefore every geodesic segment in $\M$ must be length-minimizing and \Cref{ass:assumption2} applies.

\end{example}

The above example generalizes to a natural class of concave planar billiard tables, namely the ones whose complement is a convex table.

\begin{proposition}\label{prop:concave_table}
Let $(a,b)\subset \R$ and let $\varphi\in C^{2,1}(a,b)$ satisfy $\varphi^{\prime\prime}\ge 0$. Assume that $K=\{(x,y)\in\R^2:y\le \varphi(x),\ x\in (a,b)\}$. Then for each $p_0\in \partial K$ and $\lambda\in (0,1)$ the fold $\M$ is simply connected and has non-positive Gaussian curvature.
\end{proposition}

\begin{remark}
  Notice that, since $\varphi\in C^{2,1}(a,b)$, it can be extended to a $C^2$ function defined on $\R$. Thus, we can always see $\M$ as a portion of a closed surface of $\R^3$, which is complete by the Hopf-Rinow theorem. This implies that \Cref{ass:assumption2} is satisfied.
\end{remark}

\begin{proof}[Proof of \Cref{prop:concave_table}]

Fix $p_0\in \partial K$ and $\lambda\in (0,1)$. We set $V = (a,b)\times \R$ and parameterize $\M$ through the global diffeomorphism $\Psi_\lambda:V\to \R^3$,
$$\Psi_\lambda(x,s)= \left(x,\varphi(x)-s^2,\lambda s\right).$$
Hence $M_{\lambda,\varphi}$ is diffeomorphic to $V$, and so it is simply connected.

To conclude, we pull back the Euclidean metric using the diffeomorphism $\Psi_\lambda$ and we compute the Gaussian curvature of the obtained metric. A direct computation shows that
$$\Psi_\lambda^*\g = (1+\varphi'(x)^2)\,dx^2 -4s\varphi'(x)\,dx\,ds + (4s^2+\lambda^2)\,ds^2,$$
and the unit normal vector is 
$$ N = \frac{\partial_x\Psi_\lambda\times \partial_s\Psi_\lambda} {|\partial_x\Psi_\lambda\times \partial_s\Psi_\lambda|} = \frac{(\lambda\varphi'(x),-\lambda,-2s)} {\sqrt{4s^2+\lambda^2(1+\varphi'(x)^2)}}.$$
Finally, the coefficients of the second fundamental form are
$$\langle \partial_{xx}\Psi_\lambda,N\rangle = -\frac{\lambda\varphi''(x)} {\sqrt{4s^2+\lambda^2(1+\varphi'(x)^2)}},\quad \langle \partial_{xs}\Psi_\lambda,N\rangle=0,$$
$$\langle \partial_{ss}\Psi_\lambda,N\rangle = \frac{2\lambda} {\sqrt{4s^2+\lambda^2(1+\varphi'(x)^2)}}. $$
Therefore, the Gauss equation for a Riemannian surface in $\R^3$ yields that the Gaussian curvature of $\M$ is
$$ G_{\M}(x,s)= -\frac{2\lambda^2\varphi''(x)} {(4s^2+\lambda^2(1+\varphi'(x)^2))^2}$$
which is non-positive at every $(x,s)\in \R^2$.

\end{proof}

\appendix

\section{An explicit local construction of a family of folds}

The goal of this appendix is to prove that the local construction of the folds at the beginning of section 3.2 satisfies all the technical conditions of \Cref{def:family_folds}. In what follows, we will assume that $\g$ is the Euclidean metric on $\R^{n+1}$ and $H:=\left\{(x,x_{n+1})\in \R^n\times \R:\ x_{n+1}=0\right\}$. 

Before starting, let us recall the construction of the folds. Let $K\subset H$ be a billiard table and fix $p_0\in \partial K$. We choose an open bounded neighborhood $U\subset \R^{n+1}$ of $p_0$ and a function $f\in C^{2,1}(K\cap\overline{U})$ with a regular value at $0$ such that
$$K\cap U =\{(x,0)\in H: f(x)\ge 0\},\quad \partial K \cap U = f^{-1}(0).$$
Here, we construct the family of Riemannian hypersurfaces $\left\{\left(\M,g_\lambda\right)\right\}_{\lambda\in (0,1)}$, where
\begin{equation*}
\M = \left\{(x,x_{n+1})\in \R^{n+1}:\ x_{n+1}^2 = \lambda^2f(x),\ x\in K\cap U\right\}\subset \R^{n+1}
\end{equation*}
and $g_\lambda$ is the metric on $\M$ induced by $\overline{g}$.

\begin{theorem}

The collection $\{(\M,g_\lambda)\}_{\lambda\in(0,1)}$ is a family of folds over $p_0$ in the sense of Definition \ref{def:family_folds}.

\end{theorem}

We will divide the proof into several steps. First, we notice that the first two items of \Cref{def:family_folds} are automatically satisfied for each $\M$, so it remains only to prove item 3.

Let $U^\prime\Subset U$ (i.e. $U^\prime$ is open and $\overline U^\prime \subset U$) and let $C = \pi_K^{-1}(K\cap \overline U^\prime)\subseteq D(K)$ be a compact neighborhood of $p_0$, where we recall that $\pi_K:D(K)\to K$ is the double projection. Let us introduce the function $\sigma:D(K)\to \R$ defined as follows.
If $q_+$ and $q_-$ are the two lifts in $D(K)$ of a point
$x\in \Int(K)$, we have
$$\sigma(q_+)=\sqrt{f(x)},\qquad
\sigma(q_-)=-\sqrt{f(x)},$$
while if $q_0\in \partial K\subset D(K)$, define $\sigma(q_0)=0$.
Notice that $\sigma$ is continuous on $D(K)$. Then we define the function $\Phi_\lambda:C\to M_\lambda$ as
$$\Phi_\lambda(q) = (\pi_K(q), \lambda \sigma(q)).$$
We claim that such $\Phi_\lambda$ satisfies the properties in item 3 of \Cref{def:family_folds}. 
Notice that $\Phi_\lambda$ is a homeomorphism from the topological manifold with boundary $C$ onto its image and, moreover, we observe that there is a compact neighborhood of $p_0$, $W\subset \R^{n+1}$, such that $\Phi_\lambda(C)\subseteq W$ for all $\lambda\in (0,1)$. Indeed, since $C$ is compact and $\pi_K$ and $\sigma(q)$ are continuous on $C$, there is $R>0$ such that
$$|\Phi_\lambda(q)|\le |\pi_K(q)|+\lambda|\sigma(q)|\le 2R.$$

\begin{proposition}
  We have that
  $$\lim\limits_{\lambda\to 0^+}\; |\Phi_\lambda(q)- \pi_K(q)| = 0$$
  uniformly in $C$. 
\end{proposition}

\begin{proof}
  We let $D:= \max\limits_{C}\sigma$ and notice that, for each $q\in C$,
  $$|\Phi_\lambda(q)-\pi_K(q)| = \left|\lambda\sigma(q)\right| \le \lambda\, D.$$
  In particular this implies $|\Phi_\lambda(q)-\pi_K(q)|\to 0$ as $\lambda\to 0^+$ uniformly in $C$.
\end{proof}

\begin{proposition}
  There is $T_0>0$ such that, for each $\lambda\in (0,1)$, every arclength-parameterized geodesic segment $\gamma$ in $\M$ with $\gamma(0)=p_0$ can be extended up to $[-T_0,T_0]$ and its image is contained in the interior of $\Phi_\lambda(C)$.
\end{proposition} 

\begin{proof}
Let $P:\mathbb R^{n+1}\to H$ be the Euclidean orthogonal projection,
$P(x,x_{n+1})=(x,0)$. Since $U'\Subset U$ and $p_0\in U'$, we set 
$$\delta:=\operatorname{dist}_{H}(p_0,H\setminus U')>0 \quad \mbox{and}\quad T_0 := \delta/2.$$
Fix $\lambda\in(0,1)$, and let $\gamma$ be a maximal
arclength-parameterized geodesic in $\M$ with $\gamma(0)=p_0$. We prove the statement for positive times; the negative-time argument is identical.

Let $[0,b)$ be the positive part of the maximal interval of definition
of $\gamma$. We claim first that $P(\gamma(t))\in U'$ for every
$t\in [0,\min(b,T_0))$. For the sake of contradiction, suppose that $t_*$ is the first time at which $P(\gamma(t_*))\in H\setminus U'$. Then, for every $t\in[0,t_*]$, we have $\gamma(t)\in \Phi_\lambda(C)$. Therefore, since $\gamma$ is parameterized by unit speed,
$$|(P\circ\gamma)'(t)|
\le |\gamma'(t)|
= 1.$$
It follows that
$$|P(\gamma(t_*))-p_0|
\le t_*
< T_0 =\frac{\delta}{2}.$$
But since $P(\gamma(t_*))\in H\setminus U'$, by the definition of $\delta$, we should have $|P(\gamma(t_*))-p_0|\ge \delta$, which is a contradiction. Hence $P(\gamma(t))\in U'$ for all
$t\in [0,\min(b,T_0))$. In fact, for every such $t$, the same estimate gives
$$|P(\gamma(t))-p_0|\le t\le T_0=\frac{\delta}{2}.$$
Thus \(\gamma(t)\) stays in a compact subset of
$\Int_{\M}(\Phi_\lambda(C))$ for $t\in [0,\min(b,T_0))$. To conclude, we show $b>T_0$. Indeed, if $b\le T_0$, $\gamma([0,b))$ would be contained in a compact subset of the Riemannian manifold $\M$ and the standard continuation argument for solutions of the geodesic equations would allow us to extend $\gamma$ beyond $b$, contradicting maximality.
\end{proof}
    
\begin{proposition}
  We have that
  $$\lim_{\lambda\to 0^+}\;\left|d_{g_\lambda}(\Phi_\lambda(q_1), \Phi_\lambda(q_2)) - d_{D(K)}(q_1,q_2)\right| = 0$$
  uniformly in $q_1,q_2\in C$.
\end{proposition}

\begin{proof}
  We break the proof in two steps. First, we claim that for each fixed $q_1,q_2\in C$ and a Lipschitz regular curve $\alpha:[0,1]\to C$ such that $\alpha(0)=q_1$, $\alpha(1) = q_2$ and
  $$\mbox{Var}(\sigma\circ \alpha):=\int_0^1|(\sigma\circ \alpha)^\prime(t)|\, dt< \infty$$
  we have 
  $$L_{g_\lambda}(\Phi_\lambda\circ \alpha)\to L_K(\pi_K\circ \alpha).$$
  On the one hand, we obviously have $L_K(\pi_K\circ \alpha)\le L_{g_\lambda}(\Phi_\lambda\circ \alpha)$, since $L_K$ and $L_{g_\lambda}$ are both induced by the ambient Euclidean metric $\g$ and $H$ is totally geodesic in $\R^{n+1}$. On the other hand, we observe
  \begin{equation*}
    \begin{split}
       L_{g_\lambda}(\Phi_\lambda\circ \alpha) & = \int_0^1\sqrt{|(\pi_K\circ \alpha)^\prime(t)|^2 + \lambda^2 |(\sigma\circ \alpha)^\prime(t)|^2}\, dt \\
         & \le \int_0^1|(\pi_K\circ \alpha)^\prime(t)|\, dt + \lambda\int_0^1 |(\sigma\circ \alpha)^\prime(t)|\, dt \\
         & = L_K(\pi_K\circ \alpha) + \lambda\mbox{Var}(\sigma\circ \alpha),
    \end{split}
  \end{equation*}
  which converges to $L_K(\pi_K\circ \alpha)$ as $\lambda\to 0^+$.
  
  Next, we fix $\eta>0$ and we claim that there is $B_\eta>0$ such that for each $q_1,q_2\in C$ there is a Lipschitz regular curve $\alpha:[0,1]\to C$ such that $\alpha(0)=q_1$, $\alpha(1)=q_2$ and
  $$L_K(\pi_K\circ \alpha)\le d_{D(K)}(q_1,q_2)+\eta \quad \mbox{and}\quad \mbox{Var}(\sigma\circ \alpha)<B_\eta.$$
  The existence of an $d_{D(K)}$-almost minimizing curve is guaranteed by the fact that the length structure associated to the distance function $d_{D(K)}$ is precisely $L_{D(K)}(\alpha) = L_K(\pi_K\circ \alpha)$ (see section 2.2), so it remains to prove that such an almost minimizing curve can be chosen to have bounded variation in the $\sigma$ coordinate for all possible choices of $q_1,q_2\in C$. We postpone its construction to \Cref{lem:coord_connectors} below.

  Provided that such curves exist, by taking the infimum over those and letting $\eta\to 0$ and $\lambda\to 0^+$ we obtain
  $$\lim_{\lambda\to 0^+}\; \sup_{q_1,q_2\in C}|d_{g_\lambda}(\Phi_\lambda(q_1),\Phi_\lambda(q_2))- d_{D(K)}(q_1,q_2)| = 0.$$ 
\end{proof}

\begin{lemma}\label{lem:coord_connectors}
In the context of the proof of the above proposition, for every $\eta>0$ there exists
$B_\eta<\infty$ depending only on $\eta$ and $C$ such that, for every $q_1,q_2\in C$, there is a Lipschitz curve $\alpha:[0,1]\to C$ joining $q_1$ to $q_2$ and satisfying
$$L_{D(K)}(\alpha)\le d_{D(K)}(q_1,q_2)+\eta,\qquad
\operatorname{Var}(\sigma\circ\alpha)\le B_\eta.$$
\end{lemma}

\begin{proof}
Fix $\eta>0$ and let $C_1\subset D(K)$ be compact and satisfy $C\subset C_1$, $C\cap \partial C_1 = \varnothing$. Since $C$ is compact, we may cover $C$ by finitely many coordinate charts $V_1,\ldots,V_m\Subset C_1$ with the following
property: if two points $q_1,q_2\in C$ are sufficiently close and lie in one of these $V_j$'s, then they can be connected inside that
neighborhood by a coordinate connector $\beta$ (i.e. a curve that, in coordinates, is a straight path) satisfying
$$L_{D(K)}(\beta)<\frac{\eta}{8},
\qquad
\operatorname{Var}(\sigma\circ\beta)\le A,$$
where $A<\infty$ is independent of $q_1,q_2$.
If $V_j$ is a coordinate chart away from $\partial K$, the property follows from the fact that $\sigma$ is smooth in $V_j$. Around some $q_0\in \partial K$, we use local topological coordinates $(\theta,\sigma)$ (where $\theta$ is a smooth coordinate for $\partial K$ around $q_0$ and $\sigma$ is a continuous but not differentiable coordinate) to construct local connectors of the form
$$(\theta_1+ t(\theta_2-\theta_1),\sigma_1+t(\sigma_2-\sigma_1)),$$
which clearly have finite $\sigma$-variation. We can also take the coordinate charts small enough so that the length of those coordinate connectors is less than $\eta/8$.

Now, we choose $r<\eta/8$ small enough so that every two points of $C$ with $d_{D(K)}$-distance less than $r$ can be joined by one of the above local connectors. Also, since $C$ is compact, we consider $\{a_1,\ldots,a_N\}\subset C$ such that for each $q\in C$ there is $j\in \{1,\ldots,N\}$ such that $d_{D(K)}(q,a_j)<r$. For each pair $a_i,a_j$, fix a Lipschitz curve $\gamma_{ij}:[0,1]\to C_1$ joining $a_i$ to $a_j$ such that
$$L_{D(K)}(\gamma_{ij})
\le d_{D(K)}(a_i,a_j)+\frac{\eta}{2},
\qquad
\operatorname{Var}(\sigma\circ\gamma_{ij})<\infty.$$
Since there are only finitely many pairs, we define $B_0:=\max\limits_{1\le i,j\le N} \operatorname{Var}(\sigma\circ\gamma_{ij})<\infty$.

Now let $q_1,q_2\in C$ and choose $a_i,a_j$ such that $d_{D(K)}(q_1,a_i)$, $d_{D(K)}(q_2,a_j)<r$. Let $\beta_1$ be a local connector from $q_1$ to $a_i$, and let $\beta_2$ be a local connector from $a_j$ to $q_2$. Define $\alpha$ as the concatenation of $\beta_1$, $\gamma_{ij}$ and $\beta_2$ in this order. We estimate
$$\operatorname{Var}(\sigma\circ\alpha) \le \operatorname{Var}(\sigma\circ\beta_1) + \operatorname{Var}(\sigma\circ\gamma_{ij}) + \operatorname{Var}(\sigma\circ\beta_2) \le 2A+B_0.$$
Thus we set $B_\eta:=2A+B_0$. 
To conclude that $\alpha$ is one of the curves we were looking for, we observe
\begin{equation*}
\begin{split}
L_{D(K)}(\alpha) &\le L_{D(K)}(\beta_1) + L_{D(K)}(\gamma_{ij}) + L_{D(K)}(\beta_2)\\
&\le \frac{\eta}{8} + d_{D(K)}(a_i,a_j) + \frac{\eta}{2} + \frac{\eta}{8}\\
&\le d_{D(K)}(q_1,q_2) +\frac{\eta}{4}+\frac{3\eta}{4} = d_{D(K)}(q_1,q_2) +\eta.
\end{split}
\end{equation*}
\end{proof}

\section{Billiards approximating geodesics}

The goal of this appendix is to make this paper self-contained and sketch the proof of \Cref{thm:lange_thm} which states that geodesics on the boundary of a convex billiard table are approximated by internal billiard trajectories. The details of the proof are contained in Sections 4 and 5 of \cite{Lange}.

Let $K\subset \R^n$ be the closure of an open connected subset of the Euclidean space. We say that $K$ is a \textit{convex billiard table in $\R^n$} if its boundary $\partial K$ is of class $C^{2,1}$ and for every $p_0\in \partial K$ the scalar second fundamental form at $p_0$ (associated to the inward pointing unit normal field to $\partial K$) is positive definite. 

\begin{remark}

Notice that every convex billiard table $K$ in $\R^n$ is a space of non-negative curvature according to \Cref{def:space_curvature_below}.

\end{remark}

Here, we recall several lemmas from \cite{Lange} which will play a crucial role in the proof of \Cref{thm:lange_thm}.

\begin{lemma}[= Lemma 5.5 of \cite{Lange}]\label{lemma55}

Let $K$ be a convex billiard table in $\R^n$. Then every 0-quasigeodesic of $K$ which is contained in $\partial K$ is also a 0-quasigeodesic of $\partial K$ (with respect to its intrinsic metric).

\end{lemma}

\begin{lemma}[= Lemma 5.3 of \cite{Lange}]\label{lemma53}

Let $K$ be a convex billiard table in $\R^n$. If the initial directions of a sequence of billiard trajectories converge to a tangent vector in the boundary, then the sequence of billiard trajectories converges to a continuous curve of the boundary uniformly on every compact set of times.

\end{lemma}

\begin{theorem}[= \Cref{thm:lange_thm}]

Let $K$ be a convex body in $\R^n$ whose boundary is of class $C^{2,1}$ and has a positive definite second fundamental form at every point. Let $p_0\in \partial K$ and $\{\gamma_k\}_{k\in \N}$ be a sequence of billiard trajectories of $K$ starting from $p_0$. Then, if the initial direction of $\gamma_k$ converges to a tangent vector $v\in T_{p_0}(\partial K)$, we have that the sequence $\{\gamma_k\}_{k\in \N}$ converges locally uniformly to the corresponding geodesic of the boundary.

\end{theorem}

\begin{proof}

Let $p_0\in \partial K$ and let us consider a sequence of billiard trajectories 
$$\left\{c_k: [0,\infty)\rightarrow K\right\}_{k\in \N}$$
with $c_k(0) = p_0\in \partial K$ for every $k\in \N$ and such that there is $v\in T_{p_0}(\partial K)$ for which $c_k^+(0)\rightarrow v$  in $T_{p_0}K$ as $k\rightarrow \infty$.
Here, \Cref{lemma53} implies that the sequence $\{c_k\}_{k\in \N}$ converges to some continuous curve $c:[0,\infty)\rightarrow \partial K$ uniformly on every compact set of times $[a,b]\subset [0,\infty)$ as $k\rightarrow \infty$. 

Now, let us fix $[a,b]\subset [0,\infty)$ and denote
$$\tilde c_k:= c_k\big|_{[a,b]},\quad \tilde c:= c\big|_{[a,b]}.$$
We notice that each $\tilde c_k$ is a $0$-quasigeodesic of $K$ for all $k\in \N$. Indeed, since a billiard trajectory is a concatenation of geodesic segments of $\Int(K)$ such that, at the concatenation points, the right and left tangent vectors are polar, \Cref{prop:quasigeodesic_polar} implies that each billiard trajectory in $K$ must be a $0$-quasigeodesic curve.

Then, since $\tilde c$ is the uniform limit of $\tilde c_k$, \Cref{prop:quasigeodesic_limit} implies that $\tilde c$ is $0$-quasigeodesic of $K$ as well. 
Moreover, \Cref{lemma55} implies that $\tilde c$ is a $0$-quasigeodesic of $\partial K$ with respect to its intrinsic metric. 

In conclusion, since $\partial K$ is a Riemannian manifold without boundary and a space of non-negative curvature, \Cref{prop:geodesic_quasigeodesic} yields that $\tilde c$ is a geodesic segment of $\partial K$. Hence $c:[0,\infty)\rightarrow \partial K$ is the unique geodesic of $\partial K$ starting from $p_0$ and with initial direction
$$c^+(0) = \lim_{k\rightarrow \infty} c^+_k(0) = v.$$

\end{proof}

\printbibliography

\end{document}